\documentclass[review,hidelinks,onefignum,onetabnum]{siamart250211}

\usepackage[utf8]{inputenc} 
\usepackage[T1]{fontenc}    
\usepackage{hyperref}       
\usepackage{url}            
\usepackage{booktabs}       
\usepackage{amsfonts}       
\usepackage{nicefrac}       
\usepackage{microtype}      
\usepackage{doi}
\usepackage{amsmath, amssymb, graphicx, hyperref, booktabs, siunitx, multirow, makecell}
\usepackage{listings,xcolor}
\usepackage{enumerate}

\def\T{\mathbb{T}}
\def\d{\mathrm{d}}
\def\R{\mathbb{R}}
\def\Z{\mathbb{Z}}

\usepackage{lipsum}
\usepackage{amsfonts}
\usepackage{graphicx}
\usepackage{epstopdf}
\usepackage{algorithmic}
\ifpdf
  \DeclareGraphicsExtensions{.eps,.pdf,.png,.jpg}
\else
  \DeclareGraphicsExtensions{.eps}
\fi

\newsiamremark{remark}{Remark}
\newsiamremark{assumption}{Assumption}
\newsiamremark{hypothesis}{Hypothesis}
\crefname{hypothesis}{Hypothesis}{Hypotheses}
\newsiamthm{claim}{Claim}
\newsiamremark{fact}{Fact}
\crefname{fact}{Fact}{Facts}

\headers{Stable Physics-Informed Kernel Evolution}{Hua Su, Lei Zhang, Jin Zhao}

\title{SPIKE: Stable Physics-Informed Kernel Evolution Method for Solving Hyperbolic Conservation Laws\thanks{Submitted to the editors DATE.
\funding{L.Z. was supported by the National Key Research and Development Program of China  2024YFA0919500 and National Natural Science Foundation of China (No. 12225102, T2321001, 12288101, and 12426653). J.Z. was supported in part by the National Natural Science Foundation of China (Grant No.12301520), JR25003, and the Beijing Outstanding Young Scientist Program (No.JWZQ20240101027).}}}

\author{Hua Su\thanks{Beijing International Center for Mathematical Research, Peking University, Beijing, 100871, China. (\email{suhua@pku.edu.cn}).}
\and Lei Zhang\thanks{Corresponding author. Beijing International Center for Mathematical Research, Center for Quantitative Biology, Center for Machine Learning Research, Peking University, Beijing, 100871, China. (\email{zhangl@math.pku.edu.cn}).}
\and Jin Zhao\thanks{Academy for Multidisciplinary Studies, Capital Normal University, and Beijing National Center for Applied Mathematics, Beijing, 100048, China.
  (\email{zjin@cnu.edu.cn}).}
}

\usepackage{amsopn}

\ifpdf
\hypersetup{
  pdftitle={SPIKE},
  pdfauthor={Hua Su, Lei Zhang, Jin Zhao}
}
\fi

\begin{document}
\maketitle

\begin{abstract}
We introduce the Stable Physics-Informed Kernel Evolution (SPIKE) method for numerical computation of inviscid hyperbolic conservation laws. SPIKE resolves a fundamental paradox: how strong-form residual minimization can capture weak solutions containing discontinuities.
SPIKE employs reproducing kernel representations with regularized parameter evolution, where Tikhonov regularization provides a smooth transition mechanism through shock formation, allowing the dynamics to traverse shock singularities. 
This approach automatically maintains conservation, tracks characteristics, and captures shocks satisfying Rankine-Hugoniot conditions within a unified framework requiring no explicit shock detection or artificial viscosity. 
Numerical validation across scalar and vector-valued conservation laws confirms the method's effectiveness.
\end{abstract}



\section{Introduction}

This work focuses on the numerical solution of one-dimensional inviscid \emph{hyperbolic conservation laws} (HCL) on the torus $\T=\R/\Z$:
\begin{align*}
\partial_t q + \partial_x f(q) &= 0, \quad (x,t) \in \T \times \mathbb{R}_+,\\ 
q(x,0) &= q_0(x), \quad x \in \T,
\end{align*}
where $q = q(x,t) \in \mathbb{R}^d$ denotes the conserved variables and $f: \mathbb{R}^d \to \mathbb{R}^d$ is the flux function. This partial differential equation (PDE) represents a fundamental class of models that preserve physical quantities such as mass, momentum, or energy over time. These laws are essential in applications spanning computational fluid dynamics, gas dynamics, and wave propagation \cite{EvansPDE,leveque2002finite,toro2013riemann}. More broadly, they constitute the mathematical foundation for modern science and engineering by offering a unified language to describe phenomena across diverse fields. 

A defining feature of HCL is their tendency to develop \emph{shocks}—discontinuities that form spontaneously from smooth initial data \cite{lax1973hyperbolic}. For scalar problems ($d=1$), the solution remains constant along characteristic curves defined by $\dot{x} = f'(q)$. The crossing of these characteristic curves signifies the formation of a shock, at which point the classical \emph{strong} solution ceases to exist, requiring the PDE to be interpreted in an integral or \emph{weak} sense. Across such discontinuities, the weak solution must satisfy the Rankine-Hugoniot jump condition, which determines the shock speed from the states on either side \cite{Hesthaven2018,lax1973hyperbolic,leveque1992numerical}.

The accurate numerical treatment of shocks presents a long-standing computational challenge.       
Near discontinuities, standard discretizations often produce spurious oscillations (Gibbs phenomena) or rely on numerical diffusion that smears sharp features \cite{godounov1959difference}.      
To address these issues, numerous specialized techniques have been developed, including limiters \cite{ZHONG2013397}, Essentially Non-Oscillatory (ENO) \cite{HARTEN1987231}, Weighted ENO (WENO) schemes \cite{jiang1996efficient,liu1994weighted,Shu_2020} and oscillation-eliminating discontinuous Galerkin (OEDG) \cite{peng2025oedg}. 
These approaches typically introduce carefully designed stabilization mechanisms that enforce the entropy solution while suppressing oscillations.

\subsection{Machine Learning Approaches}

Machine learning offers an alternative to classical numerical schemes, leveraging the expressive power of neural networks \cite{Cybenko1989,Hornik1989} to construct flexible, mesh-free solution representations. Approaches in this domain are primarily distinguished by the information they use for training: empirical data or governing physical laws.

Data-driven models learn input-output mappings from large datasets \cite{ModelReductionNN,duLMM2022,kovachki2023neural,li2020fourier,lu2021learning}. While powerful for in-distribution tasks, their solutions do not inherently respect physical laws, which can lead to poor generalization and extrapolation.

Physics-informed methods, conversely, provide a remedy by embedding the governing equations directly into the training objective \cite{NNforPDE1994,bing2018DRM,Lagaris712178,raissi2019physics,SIRIGNANO20181339,sun2024lift,Zhang2024}. The Physics-Informed Neural Network (PINN) framework \cite{DeepXDE,raissi2019physics} exemplifies this by representing the solution over the entire space-time domain with a neural network, $q(x,t)\approx q_{\text{NN}}(x,t;\theta)$.
The parameters $\theta$ are then optimized to minimize the PDE residual alongside losses for initial and boundary conditions.
While broadly applicable, this approach often faces training difficulties and is restricted to a fixed time interval, precluding its use in long-term forecasting.

The dynamic approach overcomes these limitations by parameterizing only the spatial dependence of the solution, $q(x,t)\approx q\big(x;\theta(t)\big)$, and evolving the parameters according to an ordinary differential equation (ODE) \cite{RONS,NeuralGalerkin,TENG,EDNN,hao2023NED,petnn,kast2023positional}:
\begin{equation*}      
    \dot{\theta}(t) = \gamma\big(\theta(t)\big), \quad \theta(0) = \theta_0.       
\end{equation*}

For a PDE of the form $\partial_t q + \mathcal{N}_x q = 0$, where $\mathcal{N}_x$ is a spatial differential operator, the Dirac-Frenkel variational principle \cite{dirac1930,frenkel1934,lubich2008} determines the parameter velocity $\gamma(\theta)$ by minimizing the instantaneous PDE residual in a least-squares sense:
\begin{equation*}
    \gamma(\theta) = \arg\min_{\gamma} \int_{\Omega} \left| \frac{\partial q}{\partial \theta} \cdot \gamma + \mathcal{N}_x q \right|^2 \d x.
\end{equation*}

While the theoretical basis for this dynamic evolutionary framework is well-established in Hilbert space settings for smooth problems \cite{feischl2024regularized,lubich2025regularized,24OTDDTO}, a paradox arises with hyperbolic conservation laws. 
Minimizing a strong-form residual is counterintuitive for capturing shocks, as the PDE is not classically defined at discontinuities. Therefore, a naive application of this principle would be expected to become singular or unstable as shocks form.
Consequently, some methods incorporate weak-form constraints \cite{deryck2024wpinns,PINNWE}, viewing them as essential to handle discontinuous solutions.

Contrary to this view, we show that this paradox can be resolved not by abandoning the strong form, but by carefully designing the solution representation and regularizing its evolution. 
This combination allows the strong-form objective to effectively guide the solution parameters through the shock-formation singularity, removing the need for explicit weak-form constraints or auxiliary shock detection techniques.

\subsection{Our Contribution}

We introduce the Stable Physics-Informed Kernel Evolution (SPIKE) method, which resolves the paradox through a purpose-built kernel representation that replaces traditional neural network parameterizations.
\begin{equation*}
    q(x,t)\approx q(x;\theta(t)) = \sum_{i=1}^N a_i(t)k\big(x,x_i(t)\big) + b(t),
\end{equation*}
where the parameters $\theta(t)=\{a_i(t), x_i(t), b(t)\}$ consist of evolving amplitudes $a_i$, adaptive positions $x_i$, and a bias term $b$. 
The parameter evolution is governed by the regularized minimization of the strong-form residual:
\begin{equation*}
    \dot{\theta} = \arg\min_{\gamma} \int_{\mathbb{T}} \left| \frac{\partial q}{\partial \theta} \cdot \gamma + \partial_x f(q) \right|^2 \d x + \lambda |\gamma|^2.
\end{equation*}

This approach appears paradoxical: we capture discontinuous weak solutions by minimizing a strong-form residual that is classically undefined at discontinuities.
We resolve this paradox by showing that appropriate kernel representation with regularized evolution automatically exhibit correct shock behavior through its intrinsic mathematical structure rather than explicit constraints or detection mechanisms.

Our main contributions are:
\begin{itemize}
    \item \textbf{Kernel-Based Framework}: We prove that dynamic kernel representations with evolving positions naturally bridge strong-form optimization and weak solution theory, automatically enforcing conservation, characteristic propagation, and Rankine-Hugoniot jump conditions without auxiliary constraints.
    \item \textbf{Computational Efficiency}: We develop linear-complexity algorithms by exploiting the kernel structure through interpolation theory, achieving computational efficiency competitive with traditional methods.
    \item \textbf{Regularization Mechanism}: We reveal that Tikhonov regularization provides a smooth transition mechanism through shock formation, allowing the optimization dynamics to traverse discontinuities that appear singular from the strong-form perspective.
\end{itemize}

The remainder of this paper is organized as follows. Section 2 introduces the SPIKE methodology, detailing the kernel representation and the evolution dynamics. Section 3 derives the efficient numerical algorithm. Section 4 provides a theoretical analysis of the method's properties, with a focus on regularization and the shock-capturing mechanism. Section 5 presents numerical experiments on several benchmark problems. Finally, Section 6 concludes with discussion and directions for future work.

\section{SPIKE Method}

In this section, we present the formulation of the SPIKE method, beginning with its motivation inspired by neural networks, and then proceeding to its specific representation and evolution dynamics.

\subsection{Motivation from Neural Networks}

The SPIKE method employs a parametric representation inspired by analyzing two-layer neural networks for periodic functions.  
Consider the standard formulation where the input $x \in \mathbb{T}$ is mapped onto the unit circle $(\cos(2\pi x), \sin(2\pi x)) \in \mathbb{S}^1$ and then processed through a fully-connected neural network with $N$ hidden neurons:
\begin{equation*}
    q(x;\theta)=\sum_{i=1}^{N} W_i^{(2)} \cdot \sigma\left(W_i^{(1)}\cdot \begin{bmatrix} \cos(2\pi x) \\ \sin(2\pi x) \end{bmatrix} + b_i^{(1)} \right) + b^{(2)}.
\end{equation*}

If the first-layer biases are set to zero ($b_i^{(1)}=0$) and $\sigma$ is the ReLU activation function, this expression admits a more compact formulation.  
By expressing the weight vector $W_i^{(1)} \in \mathbb{R}^2$ in polar coordinates as
\begin{equation*}
    W_i^{(1)} = \left(|W_i^{(1)}|\cos(2\pi x_i), |W_i^{(1)}|\sin(2\pi x_i)\right),
\end{equation*}
the network's output can be rewritten as:
\begin{align*}
    q(x;\theta)&=\sum_{i=1}^{N} W_i^{(2)} \cdot\text{ReLU}\left[|W_i^{(1)}|\cos(2\pi x_i)\cos(2\pi x) + |W_i^{(1)}|\sin(2\pi x_i)\sin(2\pi x)\right] + b\\
    &=\sum_{i=1}^{N} W_i^{(2)} \cdot |W_i^{(1)}| \cdot\,\text{ReLU}\big[\cos\big(2\pi(x - x_i)\big)\big] + b.
\end{align*}

Defining new amplitudes $a_i = W_i^{(2)}\cdot|W_i^{(1)}|$ and the basis function $\varphi_{\text{NN}}(x) = \text{ReLU}[\cos(2\pi x)]$, we obtain
\begin{equation*}
    q(x;\theta) = \sum_{i=1}^{N} a_i \varphi_{\text{NN}}(x - x_i) + b.
\end{equation*}

This derivation shows that certain neural network architectures are mathematically equivalent to linear combinations of shifted basis functions. The SPIKE method builds on this insight by replacing the neuron's activation profile $\varphi_{\text{NN}}(x)$ with a kernel $\varphi(x)$ specifically designed for representing solutions with sharp gradients and discontinuities.

\subsection{Parametric Representation and Evolution Framework}

Motivated by the preceding analysis, the SPIKE method parameterizes the solution as a sum of shifted basis functions:
\begin{equation}
    \label{rep_with_bias}
    q(x,t) \approx q\big(x;\theta(t)\big) = \sum_{i=1}^{N} a_i(t) \varphi\big(x - x_i(t)\big) + b(t),
\end{equation}
where $\varphi:\T\to\R$ is a kernel function to be specified later.
The parameters $\theta(t)$ consist of evolving amplitudes $a_i(t)$, basis positions $x_i(t)$, and a bias term $b(t)$. 
The temporal evolution of these parameters is governed by an ODE system, involving two phases:
\begin{itemize}
\item \textbf{Initialization:} The initial parameters $\theta_0$ are determined by fitting the model to the initial condition: 
\begin{equation}
    \label{SPIKE_init}
    \theta_0 = \arg\min_{\theta} \big\|q(\cdot;\theta) - q_0(\cdot)\big\|^2 .
\end{equation}
\item \textbf{Evolution:} The parameters $\theta(t)$ are evolved according to an ODE, where the time derivative $\dot{\theta}(t)$ at each instant is found by solving the regularized optimization problem:
\begin{equation}
\label{SPIKE_EDNN}
\begin{aligned}    \min_{\dot{\theta}}&\underbrace{\int_{\T}\Big|\partial_t q + \partial_x f(q)\Big|^2 \d x}_{\text{Instantaneous PDE Residual}} + \underbrace{\lambda_a \sum_{i=1}^{N} |\dot{a}_i|^2 + \lambda_x \sum_{i=1}^{N} |\dot{x}_i|^2 + \lambda_b |\dot{b}|^2}_{\text{Tikhonov Regularization}},\\
    \text{where}&\quad\partial_t q = \sum_{i=1}^{N} \Big[\dot{a}_i(t) \varphi\big(x-x_i(t)\big) - a_i(t)\dot{x}_i(t)\varphi'\big(x-x_i(t)\big)\Big] + \dot{b}(t).
\end{aligned}
\end{equation}

The optimization problem \eqref{SPIKE_EDNN} is a least-squares problem with respect to $\dot{\theta}$. Once $\dot{\theta}$ is determined, the parameters $\theta(t)$ are advanced in time using standard ODE integration techniques.
\end{itemize}

\begin{remark}[The Dual Role of Regularization]
    The Tikhonov regularization terms serve two crucial functions. 
    First, they ensure the optimization problem is strongly convex, guaranteeing a unique solution $\dot{\theta}$ at each time step. 
    Second, and more critically, regularization is essential for the long-term existence of solutions to the parameter evolution ODEs, preventing the finite-time singularities that may arise in the unregularized ($\lambda=0$) system. 
    This property ensures stability of the evolution through complex phenomena such as shock formation, which we examine further in Section \ref{Zigzag}. 
    In practice, we use a small but positive regularization (e.g., $\lambda \sim 10^{-7}$) to achieve a balance between numerical stability and solution accuracy.
\end{remark}

\subsection{Automatic Conservation Property}
\label{sec:conservation}
A significant advantage of the least-squares formulation in \eqref{SPIKE_EDNN} is its inherent preservation of conserved quantities.
To establish this property, we first observe that for any constant $c\in\mathbb{R}$:
\begin{equation*}
    \sum_{i=1}^{N} a_i(t) \varphi\big(x - x_i(t)\big) + b(t) \equiv \sum_{i=1}^{N} a_i(t) \Big[\varphi\big(x - x_i(t)\big) - c\Big] + \Big[b(t) + c \sum_{i=1}^{N} a_i(t)\Big].
\end{equation*}
This identity shows that the choice of basis function $\varphi$ is not unique for the representation \eqref{rep_with_bias}, since any constant shift can be absorbed into the bias term $b(t)$.  
Therefore, without loss of generality, we may assume that $\varphi$ has zero mean: $\int_{\mathbb{T}} \varphi(x) \, \d x = 0$. Under this assumption, the bias term acquires a clear physical interpretation:
\begin{equation}
    \label{equ:bias}
    b(t) = \int_{\mathbb{T}} q\big(x;\theta(t)\big) \, \d x.
\end{equation}

Now, the first-order optimality condition for $\dot{b}$ in \eqref{SPIKE_EDNN} yields:
\begin{equation*}
    \lambda_b \dot{b} + \int_{\mathbb{T}} \bigg(\sum_{i=1}^{N} \big[\dot{a}_i \varphi(x-x_i) - a_i \dot{x}_i \varphi'(x-x_i)\big] + \dot{b} + \partial_x f(q) \bigg) \d x = 0.
\end{equation*}

By virtue of the zero-mean property of $\varphi$ and the periodic boundary conditions, all integral terms vanish except the term containing $\dot{b}$, resulting in $(\lambda_b + 1) \dot{b} = 0$. Consequently, the bias term $b(t)$ remains constant throughout the evolution.  
Combined with \eqref{equ:bias}, this ensures exact preservation of the total quantity of the solution without requiring explicit conservation constraints.

Given this conservation property, we can simplify the representation \eqref{rep_with_bias} to:
\begin{equation}
\label{rep_wo_bias}
    q\big(x;\theta(t)\big) = \sum_{i=1}^{N} a_i(t) \varphi\big(x - x_i(t)\big) + \bar{q}_0,
\end{equation}
where $\bar{q}_0 = \int_{\T} q_0(x) \, \d x$ denotes the constant, conserved quantity determined by the initial condition.
The evolving parameter set reduces to $\theta(t) = \{a_i(t), x_i(t)\}_{i=1}^{N}$, and our subsequent focus shifts to the zero-mean component of the solution.

\subsection{Reproducing Kernel Choice}
\label{sec:varphi_choice}
For the integrals in \eqref{SPIKE_EDNN} to be well-defined, we require $\varphi\in H^1(\mathbb{T})$. Combined with the zero-mean constraint on $\varphi$, this motivates consideration of the function space:
\begin{equation}
    \label{def:RKHS}
    H^1(\T)/\R = \left\{h \in H^1(\T) : \bar{h} = \int_{\T} h(x) \d x = 0\right\}.
\end{equation}
This space is equipped with the norm $\|h\|_{H^1(\mathbb{T})/\mathbb{R}} := \|h'\|_{L^2(\mathbb{T})}$, which is equivalent to the standard $H^1$ norm through the Poincaré inequality $\|h-\bar{h}\|_{L^2(\mathbb{T})} \lesssim\|h'\|_{L^2(\mathbb{T})}$.

Our choice of basis function emerges naturally from the structure of this space:
\begin{proposition}
    \label{prop:RKHS}
    $H^1(\T)/\R$ is a reproducing kernel Hilbert space (RKHS) with the reproducing kernel $k(x,y) = \varphi(x-y)$, where $\varphi:\T\to\R$ is the piecewise quadratic function:
\begin{equation}
\label{def_varphi}
\varphi(x)=\frac{1}{2}\{x\}^2-\frac{1}{2}\{x\}+\frac{1}{12},
\end{equation}
and $\{x\}\in[0,1)$ denotes the fractional part of $x$.
\end{proposition}

\begin{proof}
    By definition \eqref{def_varphi}, $\varphi'(z) = z - \frac{1}{2}$ for $z \in (0,1)$.
    The reproducing property is verified directly: for any $x\in\mathbb{T}$ and $h \in H^1(\mathbb{T})/\mathbb{R}$, one has $k(x,\cdot)\in H^1(\mathbb{T})/\mathbb{R}$ and
    \begin{equation}
        \label{eq:reproducing}
        \begin{aligned}
            \langle k(x,\cdot), h\rangle_{H^1(\T)/\R} &= \int_{\T} \partial_y k(x,y)\cdot\partial_y h(y) \d y = -\int_{\T}\varphi'(x-y) h'(y) \d y\\
            &= -\int_0^1 \varphi'(z) h'(x-z) \d z = -\int_{0}^{1}(z-\frac12)h'(x-z)\d z\\
            &= [(z-\frac12)h(x-z)]\big|_{0}^{1}-\int_{0}^{1}h(x-z)\d(z-\frac12)=h(x).
        \end{aligned}
    \end{equation}
    The third equality uses a change of variables $z=x-y$, and the final equality follows from integration by parts and the zero-mean property of $h$.
\end{proof}

The choice of this specific kernel has two significant consequences.
First, since $H^1(\T)/\R$ is the RKHS generated by $\varphi$, the representation \eqref{rep_with_bias} possesses the universal approximation property within this space.  
Second, the piecewise quadratic form of $\varphi$ endows the representation with the inherent capacity to approximate functions exhibiting sharp transitions or discontinuities, as illustrated in Figure \ref{fig:varphi}.  
This intrinsic adaptivity makes the SPIKE method particularly effective for representing shock waves and other steep gradients.

\begin{figure}[ht]
    \centering
    \includegraphics[width=0.53\textwidth]{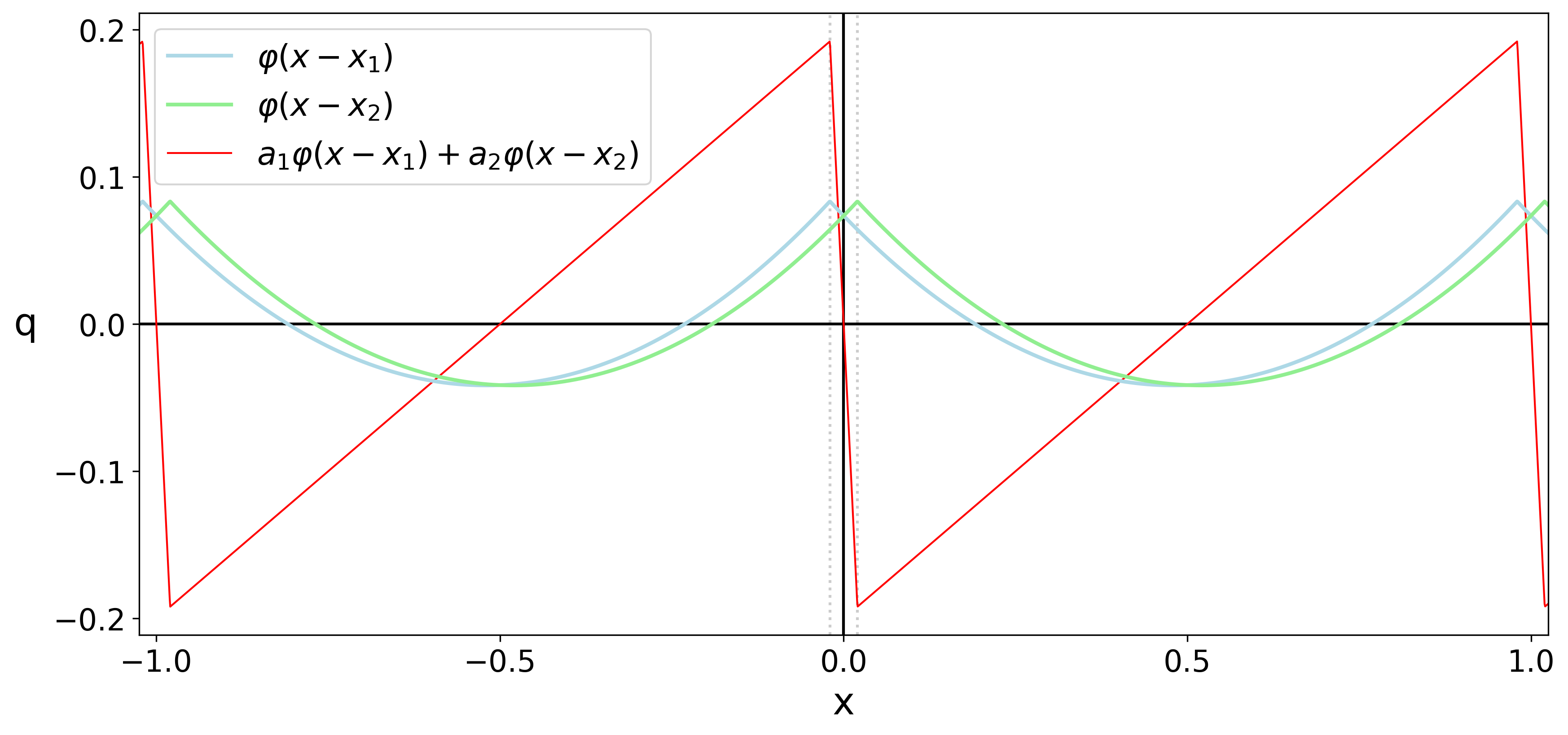}
    \caption{Linear combination of two kernels to approximate a jump discontinuity, with $a_1=-a_2=10$ and $x_1=-x_2=-0.02$. The resulting piecewise linear approximation closely captures the sharp transition.}
    \label{fig:varphi}
\end{figure}

\subsection{Linear Constraint and Spline Structure}
To further refine the model's structure and efficiency, we impose an additional linear constraint on the amplitudes:
\begin{equation}
\label{constraint}
\sum_{i=1}^N a_i = 0.
\end{equation}

This condition reveals that the kernel representation \eqref{rep_wo_bias} is, in fact, a linear spline. 
To formalize this, we periodically extend the basis positions $\{x_i\in\T\}_{i=1}^N$ to $\{x_i\in\R\}_{i\in\Z}$ by defining $x_{i+mN} = x_i + m\in[m,m+1)$ for $m\in\Z$, and assume they are arranged in strictly ascending order.

\begin{proposition}
    Under constraint \eqref{constraint}, the representation \eqref{rep_wo_bias} defines a linear spline with knots at $\{x_i\}$, and the amplitudes satisfy
\begin{equation}
\label{eq:amp_second_derivative}
-a_i = \frac{q(x_{i+1};\theta)-q(x_i;\theta)}{x_{i+1}-x_i} - \frac{q(x_i;\theta)-q(x_{i-1};\theta)}{x_i-x_{i-1}}.
\end{equation}
\end{proposition}
\begin{proof}
    By definition, the basis $\varphi$ in \eqref{def_varphi}, as well as the linear combination in \eqref{rep_wo_bias}, is piecewise quadratic. Since $\varphi''(x) = 1$ almost everywhere (except at integer points), it follows that $\partial_x^2 q(x;\theta) = \sum_{i=1}^N a_i \varphi''(x-x_i) = 0$ almost everywhere except at the basis positions. Hence, $q(x;\theta)$ is piecewise linear with breakpoints at $\{x_i\}$.

    Moreover, since $\varphi'(x)$ is continuous at $x \notin \mathbb{Z}$, but has jump discontinuities with $\varphi'(0^+) - \varphi'(0^-) = -1$ (where $0^\pm$ denotes the left and right limits), it follows that $\partial_x q(x;\theta)$ exhibits jump discontinuities with $\partial_x q(x_i^+;\theta) - \partial_x q(x_i^-;\theta) = -a_i$. The piecewise linearity of $q(x;\theta)$ then implies that
    $$\partial_x q(x_i^\pm;\theta) = \frac{q(x_{i\pm1};\theta) - q(x_i;\theta)}{x_{i\pm1} - x_i},$$ which leads directly to \eqref{eq:amp_second_derivative}.
\end{proof}

This proposition illuminates the mechanism by which the SPIKE representation \eqref{rep_wo_bias} captures solution structures.  
In smooth regions, the amplitudes $a_i$ correspond to finite difference approximations of the function's second derivative.  
Near discontinuities, two knots can be positioned close together with large, opposing amplitudes to form a sharp jump, as illustrated in Figure \ref{fig:varphi}.

\begin{remark}
    The SPIKE method exhibits structural similarities to moving mesh methods \cite{MovingMesh}, as both represent the solution on a set of adaptive grid points. However, a fundamental distinction exists in the mechanism driving the mesh movement. 
    Traditional moving mesh methods typically rely on predefined adaptation strategies, such as equidistribution principles or monitor functions, which are often independent of the governing PDE's structure. These strategies may be based on solution gradients or curvature, but generally do not directly incorporate the underlying PDE dynamics.

    In contrast, SPIKE can be interpreted as a \emph{physics-informed moving mesh} method. The evolution of the knots $\{x_i(t)\}$ is determined directly by the minimization of the PDE residual \eqref{SPIKE_EDNN}. This means the mesh motion is intrinsically coupled to the physics of the problem. As subsequent sections will show, this approach naturally aligns the knot movement with the characteristic structure of the hyperbolic conservation law, enabling effective shock capturing.
\end{remark}

\section{Numerical Implementation}
\label{fastsolver}
This section details the efficient algorithm for solving the optimization problems central to the SPIKE method. By exploiting the mathematical structure of the reproducing kernel $\varphi$ defined in \eqref{def_varphi}, we derive a fast solver that reduces the computational complexity for both the initialization and evolution steps to $O(N)$, a significant improvement over the $O(N^3)$ cost of a naive implementation.

\subsection{Initialization}
\label{sec:init}
For the initialization phase, we adopt a simplified yet effective strategy. Instead of performing a full non-linear optimization over both knot positions $\{x_i\}$ and amplitudes $\{a_i\}$, we pre-select a fixed set of knots and solve only for the amplitudes to fit the initial condition:
\begin{equation*}
    \min_{a_i\in\R^d}\bigg\|\Big(\sum_{i=1}^N a_i\varphi(\cdot-x_i)+\bar{q}_0\Big)-q_0\bigg\|_{H^1(\T)/\R}^2\quad \text{s.t. } \sum_{i=1}^N a_i = 0.
\end{equation*}

The Lagrangian for this constrained optimization problem is given by
\begin{equation*}
    \mathcal{L}(a_i,\alpha)=\bigg\|\sum_{i=1}^N a_i\varphi(\cdot-x_i)\,-(q_0-\bar{q}_0)\bigg\|_{H^1(\T)/\R}^2
    +2\alpha\cdot\sum_{i=1}^N a_i,
\end{equation*}
where $2\alpha\in\R^d$ is the Lagrange multiplier associated with the linear constraint.

Leveraging the reproducing property established in Proposition \ref{prop:RKHS}, we have:
\begin{equation*}
    \Bigl\langle \varphi(\cdot-x_i), \varphi(\cdot-x_j)\Bigr\rangle_{H^1(\T)/\R} = \varphi(x_i-x_j),\quad
    \Bigl\langle \varphi(\cdot-x_i), q_0-\bar{q}_0\Bigr\rangle_{H^1(\T)/\R} = q_0(x_i)-\bar{q}_0,
\end{equation*}

The Lagrangian then expands to:
\begin{equation*}
    \mathcal{L}(a_i,\alpha)=\sum_{i,j=1}^N \varphi(x_i-x_j)a_i\cdot a_j
    -2\sum_{i=1}^N a_i\cdot\left(q_0(x_i)-\bar{q}_0\right)
    +\Big\|q_0-\bar{q}_0\Big\|_{H^1(\T)/\R}^2
    +2\alpha\cdot\sum_{i=1}^N a_i.
\end{equation*}

The first-order optimality conditions yield:
\begin{align}
    \label{opt_a_init}
    \frac12\frac{\partial \mathcal{L}}{\partial a_i}=&\sum_{j=1}^N \varphi(x_i-x_j)a_j-\bigl(q_0(x_i)-\bar{q}_0\bigr)+\alpha=0,
\end{align}
where we can deduce that $q(x_i;\theta_0) = q_0(x_i)-\alpha$ at the knot points. Combining this with equation \eqref{eq:amp_second_derivative} yields:
\begin{equation}
    a_i = -\Big(\frac{q_0(x_{i+1})-q_0(x_i)}{x_{i+1}-x_i}-\frac{q_0(x_{i})-q_0(x_{i-1})}{x_{i}-x_{i-1}}\Big).
\end{equation}

This expression provides a direct, closed-form formula for the amplitudes based on finite difference approximations of the second derivative of the initial condition at the knot points.
In practice, when no prior information about the initial condition structure is available, we select a uniform grid $x_i = i / N$, which simplifies implementation and ensures uniform initial coverage of the domain.

\subsection{Evolution}
The evolution phase requires solving the optimization problem \eqref{SPIKE_EDNN} subject to the linear constraint $\sum_{i=1}^N \dot{a}_i = 0$ at each time step.
This presents a more sophisticated challenge due to the coupling between amplitude and knot dynamics.

\subsubsection{Optimality Conditions}
We begin by deriving the first-order optimality conditions for the constrained optimization problem.
As established in Section \ref{sec:conservation}, the bias is constant ($\dot{b}\equiv0$), so the Lagrangian for the constrained problem on $\{\dot{a}_i, \dot{x}_i\}$ is:
\begin{equation}
    \label{lagrangian_evolution}
    \begin{aligned}
        \mathcal{L}(\dot{a}_i,\dot{x}_i,\alpha)=&\int_{\T}\bigg|\sum_{i=1}^{N} \big[\varphi(x-x_i) \dot{a}_i - a_i\varphi'(x-x_i) \dot{x}_i\big] + \partial_x f(q)\bigg|^2 \d x \\
        &+ \lambda_a \sum_{i=1}^{N} |\dot{a}_i|^2 + \lambda_x \sum_{i=1}^{N} |\dot{x}_i|^2+2\alpha\cdot\sum_{i=1}^N\dot{a}_i,
    \end{aligned}
\end{equation}
where $2\alpha\in\R^d$ is the Lagrange multiplier associated for the constraint on $\dot{a}_i$.

To facilitate the following derivation, we introduce the autocorrelation function:
\begin{equation}
\label{def_g}
    g(x):=\int_{\T}\varphi(x+y)\varphi(y)\d y=-\frac{1}{24}\{x\}^4+\frac{1}{12}\{x\}^3-\frac{1}{24}\{x\}^2+\frac{1}{720},
\end{equation}
where $\{x\}\in[0,1)$ denotes the fractional part of $x \in \T$.

Expanding the squared residual term in \eqref{lagrangian_evolution} and expressing the resulting integrals using $g_{ij}=g(x_i-x_j)$, $g'_{ij}=g'(x_i-x_j)$, and $g''_{ij}=g''(x_i-x_j)$ yields:
\begin{align*}
    \mathcal{L}(\dot{a}_i,\dot{x}_i,\alpha)=&\sum_{i,j=1}^{N}\Big[g_{ij}\dot{a}_i\cdot\dot{a}_j-g'_{ij}\dot{a}_i\cdot a_j\dot{x}_j+g'_{ij}\dot{a}_j\cdot a_i\dot{x}_i-g''_{ij}a_i\dot{x}_i\cdot a_j\dot{x}_j\Big]\\
    &+2\sum_{i=1}^N\Big[\dot{a}_i\cdot\int_{\T}\varphi(x-x_i)\partial_x f(q)\d x-a_i\dot{x}_i\cdot\int_{\T}\varphi'(x-x_i)\partial_x f(q)\d x\Big]\\
    &+\lambda_a \sum_{i=1}^{N} |\dot{a}_i|^2+\lambda_x \sum_{i=1}^{N} |\dot{x}_i|^2+2\alpha\cdot\sum_{i=1}^N\dot{a}_i+\int_{\T}|\partial_x f(q)|^2\d x.
\end{align*}

The first-order optimality conditions with respect to $\dot{a}_i$ and $\dot{x}_i$ are:
\begin{equation}
    \label{opt_a_x}
    \begin{aligned}
        \frac12\frac{\partial \mathcal{L}}{\partial \dot{a}_i}&=\sum_{j=1}^N\left[g_{ij}\dot{a}_j-g'_{ij}a_j\dot{x}_j\right]+\alpha+\int_{\T}\varphi(x-x_i)\partial_x f(q)\d x+\lambda_a\dot{a}_i=0,\\
        \frac12\frac{\partial \mathcal{L}}{\partial \dot{x}_i}&=a_i\cdot\sum_{j=1}^N\left[g'_{ij}\dot{a}_j-g''_{ij}a_j\dot{x}_j\right]-a_i\cdot\int_{\T}\varphi'(x-x_i)\partial_x f(q)\d x+\lambda_x\dot{x}_i=0.
    \end{aligned}
\end{equation}

Rather than solving this dense linear system directly, which incurs $O(N^3)$ computational complexity, we introduce a key reformulation through two auxiliary functions:
\begin{align}
    \label{def_F}
    \mathfrak{F}(\xi) &= -\int_{\T}\varphi(x-\xi)\partial_x f(q) \d x,\\
    \label{def_h}
    \mathfrak{h}(\xi) &= \sum_{j=1}^N\Big[\dot{a}_j g(\xi-x_j)-a_j\dot{x}_j g'(\xi-x_j)\Big]+\alpha.
\end{align}

The key insight behind this reformulation is that $\mathfrak{F}(\xi)$ depends solely on the current solution $q\big(x;\theta(t)\big)$ and is computable via numerical integration, while $\mathfrak{h}(\xi)$ encapsulates all unknown parameters $\{\dot{a}_i, \dot{x}_i, \alpha\}$ in a structured representation. This separation allows us to rewrite the optimality conditions \eqref{opt_a_x} more elegantly:  
\begin{equation}  
    \label{opt_a_x_new}  
    \begin{aligned}  
        \frac12\frac{\partial \mathcal{L}}{\partial \dot{a}_i} &= \mathfrak{h}(x_i)-\mathfrak{F}(x_i)+\lambda_a\dot{a}_i = 0,\\  
        \frac12\frac{\partial \mathcal{L}}{\partial \dot{x}_i} &= a_i\cdot\mathfrak{h}'(x_i)-a_i\cdot\mathfrak{F}'(x_i)+\lambda_x\dot{x}_i = 0.
    \end{aligned}  
\end{equation}

This reformulation establishes a clear connection between the known function $\mathfrak{F}$ and the unknown function $\mathfrak{h}$ at the knot points $\{x_i\}$, providing a structured pathway to determine the unknown parameters.

To simplify further, we define a difference function $\mathfrak{D}(\xi):=\mathfrak{F}(\xi)-\mathfrak{h}(\xi)$. The optimality conditions then take the remarkably simple form:
\begin{equation}
    \label{opt_a_x_F}
    \begin{aligned}
        \mathfrak{D}(x_i) &= \lambda_a \dot{a}_i,\\
        a_i \cdot \mathfrak{D}'(x_i) &= \lambda_x \dot{x}_i.
    \end{aligned}
\end{equation}

These conditions reveal that once we determine $\mathfrak{D}(x_i)$ and $\mathfrak{D}'(x_i)$ at the knots, we can directly compute $\dot{a}_i$ and $\dot{x}_i$. 
The heart of our fast solver is the demonstration that the stacked vectors $Z_i := \begin{bmatrix}\mathfrak{D}(x_i)\\\mathfrak{D}'(x_i)\end{bmatrix}\in\R^{2d}$ satisfy a block tridiagonal linear system with periodic boundary conditions:
\begin{equation}
\label{block_tridiag_full}
    \begin{bmatrix}
        \mathbb{G}_{1,1} & \mathbb{G}_{1,2} & 0 & \cdots & 0 & \mathbb{G}_{1,N}\\
        \mathbb{G}_{2,1} & \mathbb{G}_{2,2} & \mathbb{G}_{2,3} & \cdots & 0 & 0\\
        0 & \mathbb{G}_{3,2} & \mathbb{G}_{3,3} & \cdots & 0 & 0\\
        \vdots & \vdots & \vdots & \ddots & \vdots & \vdots\\
        0 & 0 & 0 & \cdots & \mathbb{G}_{N-1,N-1} & \mathbb{G}_{N-1,N}\\
        \mathbb{G}_{N,1} & 0 & 0 & \cdots & \mathbb{G}_{N,N-1} & \mathbb{G}_{N,N}
    \end{bmatrix}
    \begin{bmatrix}
        Z_1\\
        Z_2\\
        \vdots\\
        Z_N
    \end{bmatrix}
    =
    \begin{bmatrix}
        R_1\\
        R_2\\
        \vdots\\
        R_N
    \end{bmatrix}
\end{equation}
The block matrices $\mathbb{G}_{i,j}\in\R^{2d\times 2d}$ and right-hand side vectors $R_i\in\R^{2d}$ are derived in the following subsections.
This system can be solved in $O(N)$ time using standard algorithms for such structures, such as the SPIKE algorithm \cite{MatrixComputations,POLIZZI2006177}.

\subsubsection{Cubic Spline Structure}

The derivation of the block tridiagonal system relies on understanding the special mathematical structure of $\mathfrak{h}(\xi)$ defined in \eqref{def_h}.

From the definition \eqref{def_g}, we observe that $g(x)$ is a piecewise quartic polynomial with $g^{(4)}(x)=-1$ almost everywhere (except at integer points). Combined with the definition \eqref{def_h} and the constraint $\sum_{i=1}^N\dot{a}_i=0$, this implies that $\mathfrak{h} \in C^1(\mathbb{T})$ and $\mathfrak{h}^{(4)}(x)=-\sum_{j=1}^N\dot{a}_j=0$ almost everywhere (except at the knots).     
    
Consequently, $\mathfrak{h}$ is a cubic polynomial on each subinterval between knots with $C^1$ continuity at the knots—in other words, a cubic Hermite spline with knots at $\{x_i\}$. This spline structure is the cornerstone of our efficient algorithm.

Let $\delta_{i_\pm} = |x_{i\pm1}-x_i|$ denote the distances between knot $x_i$ and its adjacent neighbors, and let $\mathfrak{h}_i = \mathfrak{h}(x_i)$ and $\mathfrak{h}'_i = \mathfrak{h}'(x_i)$ denote the values and first derivatives at the knots.
Classical spline theory \cite{Gautschi1997numerical} establishes that $\mathfrak{h}$ is uniquely determined by these nodal values:
\begin{equation}
    \label{h_spline}
    \mathfrak{h}(x) = p_{00}(s) \mathfrak{h}_i + p_{10}(s) (x_{i+1}-x_i) \mathfrak{h}'_i + p_{01}(s) \mathfrak{h}_{i+1} + p_{11}(s) (x_{i+1}-x_i) \mathfrak{h}'_{i+1},
\end{equation}
for $x\in[x_i,x_{i+1}]$ and $s=\frac{x-x_i}{x_{i+1}-x_i}\in[0,1]$. The basis functions are given by:
\begin{equation*}
    p_{00}(s) = 2s^3 - 3s^2 + 1,\; p_{10}(s) = s(s-1)^2,\; p_{01}(s) = -2s^3 + 3s^2,\; p_{11}(s) = s^2(s-1).
\end{equation*}

Furthermore, the definition \eqref{def_g} indicates that $g^{(3)}(x)$ has jump discontinuities at integer points, with $g^{(3)}(0^+) - g^{(3)}(0^-) = 1$. This translates to jump conditions for the derivatives of $\mathfrak{h}$ at the knots:
\begin{align*}
    \dot{a}_i &= \mathfrak{h}^{(3)}(x_i^+) - \mathfrak{h}^{(3)}(x_i^-),\\
    -a_i\dot{x}_i &= \mathfrak{h}^{(2)}(x_i^+) - \mathfrak{h}^{(2)}(x_i^-).
\end{align*}

These jump conditions provide the crucial link between the unknown parameters and the spline representation. Using the interpolation formula \eqref{h_spline}, we can express these conditions in terms of the nodal values.
\begin{align}
    \label{hDDDjump_expanded}
    \dot{a}_i &= \bigg(\frac{12(\mathfrak{h}_i-\mathfrak{h}_{i+1})}{\delta_{i_+}^3} + \frac{6(\mathfrak{h}'_i+\mathfrak{h}'_{i+1})}{\delta_{i_+}^2}\bigg) - \bigg(\frac{12(\mathfrak{h}_{i-1}-\mathfrak{h}_i)}{\delta_{i_-}^3} + \frac{6(\mathfrak{h}'_{i-1}+\mathfrak{h}'_i)}{\delta_{i_-}^2}\bigg)\\
    \label{hDDjump_expanded}
    -a_i\dot{x}_i &= \bigg(\frac{6(\mathfrak{h}_{i+1}-\mathfrak{h}_i)}{\delta_{i_+}^2} - \frac{(4\mathfrak{h}'_{i}+2\mathfrak{h}'_{i+1})}{\delta_{i_+}}\bigg) - \bigg(\frac{6(\mathfrak{h}_{i-1}-\mathfrak{h}_{i})}{\delta_{i_-}^2} + \frac{(4\mathfrak{h}'_{i}+2\mathfrak{h}'_{i-1})}{\delta_{i_-}}\bigg)
\end{align}

\subsubsection{Block Tridiagonal Linear System}
Having established the cubic spline structure of $\mathfrak{h}$, we now derive the block tridiagonal linear system \eqref{block_tridiag_full} by expressing the jump conditions \eqref{hDDDjump_expanded}\eqref{hDDjump_expanded} in terms of $\mathfrak{D}_i = \mathfrak{D}(x_i)$ and $\mathfrak{D}'_i = \mathfrak{D}'(x_i)$.

Substituting the relations \eqref{opt_a_x_F} into the left-hand sides of the jump conditions, and expressing $\mathfrak{h}_i = \mathfrak{F}_i - \mathfrak{D}_i$ and $\mathfrak{h}'_i = \mathfrak{F}'_i - \mathfrak{D}'_i$ into the right-hand sides, we obtain:
\begin{align*}
    \frac{\mathfrak{D}_i}{\lambda_a} &= \bigg(\frac{12(\mathfrak{F}_i-\mathfrak{F}_{i+1})}{\delta_{i_+}^3} + \frac{6(\mathfrak{F}'_i+\mathfrak{F}'_{i+1})}{\delta_{i_+}^2}\bigg) - \bigg(\frac{12(\mathfrak{F}_{i-1}-\mathfrak{F}_i)}{\delta_{i_-}^3} + \frac{6(\mathfrak{F}'_{i-1}+\mathfrak{F}'_i)}{\delta_{i_-}^2}\bigg)\\
    -& \bigg(\frac{12(\mathfrak{D}_i-\mathfrak{D}_{i+1})}{\delta_{i_+}^3} + \frac{6(\mathfrak{D}'_i+\mathfrak{D}'_{i+1})}{\delta_{i_+}^2}\bigg) + \bigg(\frac{12(\mathfrak{D}_{i-1}-\mathfrak{D}_i)}{\delta_{i_-}^3} + \frac{6(\mathfrak{D}'_{i-1}+\mathfrak{D}'_i)}{\delta_{i_-}^2}\bigg),\\
    -a_i\frac{a_i\cdot\mathfrak{D}'_i}{\lambda_x} &= \bigg(\frac{6(\mathfrak{F}_{i+1}-\mathfrak{F}_i)}{\delta_{i_+}^2} - \frac{(4\mathfrak{F}'_{i}+2\mathfrak{F}'_{i+1})}{\delta_{i_+}}\bigg) - \bigg(\frac{6(\mathfrak{F}_{i-1}-\mathfrak{F}_{i})}{\delta_{i_-}^2} + \frac{(4\mathfrak{F}'_{i}+2\mathfrak{F}'_{i-1})}{\delta_{i_-}}\bigg)\\
    -& \bigg(\frac{6(\mathfrak{D}_{i+1}-\mathfrak{D}_i)}{\delta_{i_+}^2} - \frac{(4\mathfrak{D}'_{i}+2\mathfrak{D}'_{i+1})}{\delta_{i_+}}\bigg) + \bigg(\frac{6(\mathfrak{D}_{i-1}-\mathfrak{D}_{i})}{\delta_{i_-}^2} + \frac{(4\mathfrak{D}'_{i}+2\mathfrak{D}'_{i-1})}{\delta_{i_-}}\bigg).
\end{align*}

Rearranging terms to place all unknowns ($\mathfrak{D}_i, \mathfrak{D}'_i$) on one side yields the desired block tridiagonal system \eqref{block_tridiag_full}, where the $i$-th block row corresponds to the equations for knot $x_i$:
\begin{equation*}
    \mathbb{G}_{i,i-1}Z_{i-1} + \mathbb{G}_{i,i}Z_i + \mathbb{G}_{i,i+1}Z_{i+1} = R_i.
\end{equation*}

The coefficient block matrices $\mathbb{G}_{i,j}\in\mathbb{R}^{2d\times 2d}$ are given by:
\begin{align*}
    \mathbb{G}_{i,i} &= \begin{bmatrix}
        12\big(1/\delta_{i_+}^{3}+1/\delta_{i_-}^{3}\big)&6\big(1/\delta_{i_+}^{2}-1/\delta_{i_-}^{2}\big)\\
        6\big(1/\delta_{i_+}^{2}-1/\delta_{i_-}^{2}\big)&4\big(1/\delta_{i_+}+1/\delta_{i_-}\big)
    \end{bmatrix}\otimes\mathbb{I}_d+\begin{bmatrix}
        \frac{1}{\lambda_a}\mathbb{I}_d&0\\
        0&\frac{1}{\lambda_x}a_ia_i^T
    \end{bmatrix},\\
    \mathbb{G}_{i,i-1} &= \begin{bmatrix}
        -12/\delta_{i_-}^{3}&-6/\delta_{i_-}^{2}\\
        6/\delta_{i_-}^{2}&2/\delta_{i_-}
    \end{bmatrix}\otimes\mathbb{I}_d,\quad
    \mathbb{G}_{i,i+1} = \begin{bmatrix}
        -12/\delta_{i_+}^{3}&6/\delta_{i_+}^{2}\\
        -6/\delta_{i_+}^{2}&2/\delta_{i_+}
    \end{bmatrix}\otimes\mathbb{I}_d.
\end{align*}
where $\mathbb{I}_d$ denotes the identity matrix, and $\otimes$ denotes the Kronecker product.

The right-hand side vectors $R_i \in \mathbb{R}^{2d}$ are expressed as:
\begin{equation}
    \label{def_Ri}
    R_i = \begin{bmatrix}
        \Big(\frac{12(\mathfrak{F}_i-\mathfrak{F}_{i+1})}{\delta_{i_+}^3} + \frac{6(\mathfrak{F}'_i+\mathfrak{F}'_{i+1})}{\delta_{i_+}^2}\Big) - \Big(\frac{12(\mathfrak{F}_{i-1}-\mathfrak{F}_i)}{\delta_{i_-}^3} + \frac{6(\mathfrak{F}'_{i-1}+\mathfrak{F}'_i)}{\delta_{i_-}^2}\Big)\\
        \Big(\frac{6(\mathfrak{F}_i-\mathfrak{F}_{i+1})}{\delta_{i_+}^2} + \frac{(4\mathfrak{F}'_{i}+2\mathfrak{F}'_{i+1})}{\delta_{i_+}}\Big) - \Big(\frac{6(\mathfrak{F}_{i}-\mathfrak{F}_{i-1})}{\delta_{i_-}^2} - \frac{(4\mathfrak{F}'_{i}+2\mathfrak{F}'_{i-1})}{\delta_{i_-}}\Big)
    \end{bmatrix}.
\end{equation}

The computational algorithm proceeds in three efficient stages: first, numerically evaluating the right-hand side vectors $R_i$; second, solving the block tridiagonal system \eqref{block_tridiag_full} to obtain $\mathfrak{D}_i$ and $\mathfrak{D}'_i$; and finally, recovering the temporal derivatives $\dot{a}_i$ and $\dot{x}_i$ using the simple relations in \eqref{opt_a_x_F}.

\subsubsection{Efficient Evaluation of \texorpdfstring{$R_i$}{Ri}}
The final element of the fast solver is an efficient method for computing the right-hand side vectors $R_i$. Although $\mathfrak{F}(\xi)$ is defined by an integral over the whole domain, the terms required for $R_i$ can be computed using only local information.

We begin by introducing the notation:
\begin{equation}
    \label{def_Fhat}
    \begin{aligned}
        \widehat{\mathfrak{F}}^{(3)}_{i+1/2} &:= \frac{12(\mathfrak{F}_i-\mathfrak{F}_{i+1})}{\delta_{i_+}^3} + \frac{6(\mathfrak{F}'_i+\mathfrak{F}'_{i+1})}{\delta_{i_+}^2},\qquad
        \widehat{\mathfrak{F}}^{(2)}_{i+1/2} := \frac{\mathfrak{F}'_{i+1} - \mathfrak{F}'_i}{\delta_{i_+}},
    \end{aligned}
\end{equation}
which allows the right-hand side $R_i$ to be compactly written as:
\begin{equation*}
    R_i = \begin{bmatrix}
        \widehat{\mathfrak{F}}_{i+1/2}^{(3)} - \widehat{\mathfrak{F}}_{i-1/2}^{(3)}\\
        -\left(\widehat{\mathfrak{F}}_{i+1/2}^{(2)} - \widehat{\mathfrak{F}}_{i+1/2}^{(3)}\cdot\delta_{i_+}/2\right) + \left(\widehat{\mathfrak{F}}_{i-1/2}^{(2)} + \widehat{\mathfrak{F}}_{i-1/2}^{(3)}\cdot\delta_{i_-}/2\right)
    \end{bmatrix}.
\end{equation*}

The reproducing property of $\varphi$ established in Proposition \ref{prop:RKHS} provides a key simplification:
\begin{equation}
    \label{equ:F'}
    \mathfrak{F}'(\xi) = \int_{\T} \varphi'(\xi - x) \partial_x f(q(x)) \d x = f(q(\xi)) - \overline{f(q)},
\end{equation}
revealing that $\mathfrak{F}$ serves as an antiderivative of $f(q)$ up to an additive constant.

Using integration by parts, we find:
\begin{align*}
    &\int_0^1 6s(1-s)\cdot\mathfrak{F}'''(x_i+s\delta_{i_+})\,\d s = \int_0^1 6s(1-s)\d\left[\mathfrak{F}''(x_i+s\delta_{i_+})/\delta_{i_+}\right]\\
    =& -\left[6s(1-s)\mathfrak{F}''(x_i+s\delta_{i_+})/\delta_{i_+}\right]_0^1 - \int_0^1 6(1-2s)\mathfrak{F}''(x_i+s\delta_{i_+})/\delta_{i_+}\,\d s\\
    =&\;0-\left[6(1-2s)\mathfrak{F}'(x_i+s\delta_{i_+})/\delta_{i_+}^2\right]_0^1 - 12\int_0^1 \mathfrak{F}'(x_i+s\delta_{i_+})/\delta_{i_+}^2\,\d s\\
    =&\;6\big(\mathfrak{F}'(x_i+\delta_{i_+})+\mathfrak{F}'(x_i)\big)/\delta_{i_+}^2 - 12\big(\mathfrak{F}(x_i+\delta_{i_+})-\mathfrak{F}(x_i)\big)/\delta_{i_+}^3 = \widehat{\mathfrak{F}}_{i+\frac12}^{(3)}.
\end{align*}

Since our solution representation $q(x)$ is piecewise linear between consecutive knots, combined with \eqref{equ:F'}, we have:
\begin{equation*}
    \mathfrak{F}'''(x_i+s\delta_{i_+}) = \partial_s^2 \big(q(x_i+s\delta_{i_+})\big)/\delta_{i_+}^2 = \partial_s^2 f\big(q_i + s(q_{i+1}-q_i)\big)/\delta_{i_+}^2,
\end{equation*}
which can be evaluated analytically given the explicit form of the flux function $f$.

Combining these results, the components of $R_i$ reduce to expressions that depend only on local quantities:
\begin{equation*}
    \begin{aligned}
        \widehat{\mathfrak{F}}_{i+1/2}^{(3)} &= \int_0^1 6s(1-s)\cdot\partial_s^2 f\big(q_i + s(q_{i+1}-q_i)\big)/\delta_{i_+}^2\;\d s,\\
        \widehat{\mathfrak{F}}_{i+1/2}^{(2)} &= \big(f(q_{i+1}) - f(q_{i})\big)/\delta_{i_+}.
    \end{aligned}
\end{equation*}

These integral expressions can be evaluated accurately using standard numerical quadrature methods, applied simultaneously to all subintervals between knots. 
The key advantage is that the entire right-hand side computation becomes highly efficient, requiring only local evaluations rather than global integrations. 
This completes the derivation of the fully fast solver, transforming a seemingly complex and dense optimization problem into a sparse, efficiently solvable linear system.

\section{Theoretical Perspective for Scalar Conservation Laws}
\label{sec:analysis}

This section addresses the central conceptual challenge of the SPIKE method: how can an algorithm that minimizes the strong-form PDE residual correctly capture weak solutions containing shocks? 
From a classical viewpoint, the strong-form PDE residual is undefined at discontinuities, which raises fundamental questions about the stability and coherence of the parameter evolution as shocks form.
Nevertheless, our numerical results show that the method robustly handles shock dynamics.

We resolve this paradox with a two-part analysis.
First, we use an analytical example to show how Tikhonov regularization maintains smooth and stable parameter evolution that bridges pre-shock and post-shock dynamics.
Second, we analyze the adaptive behavior of the kernel positions, showing how the unified optimization framework naturally transitions from tracking characteristics in smooth regions to satisfying the Rankine-Hugoniot jump condition across shocks.
These findings show that the combination of a dynamic representation and a regularized evolution law overcomes the classical limitations of strong-form methods.

\subsection{Regularization Analysis}
\label{Zigzag}
To understand how regularization enables shock capturing, we conduct a precise mathematical analysis using an exact solution to Burgers' equation $\partial_t u + \partial_x(u^2/2) = 0$. This approach eliminates approximation errors and allows us to isolate the pure effects of regularization on parameter dynamics.

The exact zigzag solution is a piecewise linear profile given by:
\begin{equation}
    \label{true_zigzag_solution}
    u^*(x,t) = \begin{cases}
        -H^*(t) \cdot \frac{x}{\delta^*(t)}, & x \in (-\delta^*(t), \delta^*(t)), \\
        -H^*(t) \cdot \frac{2x-1}{2\delta^*(t)-1}, & x \in (\delta^*(t), 1-\delta^*(t)),
    \end{cases}
\end{equation}
where the shape parameters evolve according to:
\begin{equation}
    \label{true_zigzag}
    H^*(t) = \frac{H_0}{1+2\max(H_0t-\delta_0, 0)}, \quad
    \delta^*(t) = \max(\delta_0-H_0t, 0),
\end{equation}
with initial conditions $H_0 = 1$ and $\delta_0 = 1/4$. 

This solution exhibits the classic shock formation scenario: the profile steadily steepens until $t = \delta_0/H_0 = 1/4$, at which point $\delta^*(t)$ reaches zero and a shock forms. The evolution is illustrated in Figure \ref{fig:zigzag}, where we can observe both the exact solution and SPIKE approximations with different regularization strengths.
\begin{figure}[ht]
    \centering
    \includegraphics[width=\textwidth]{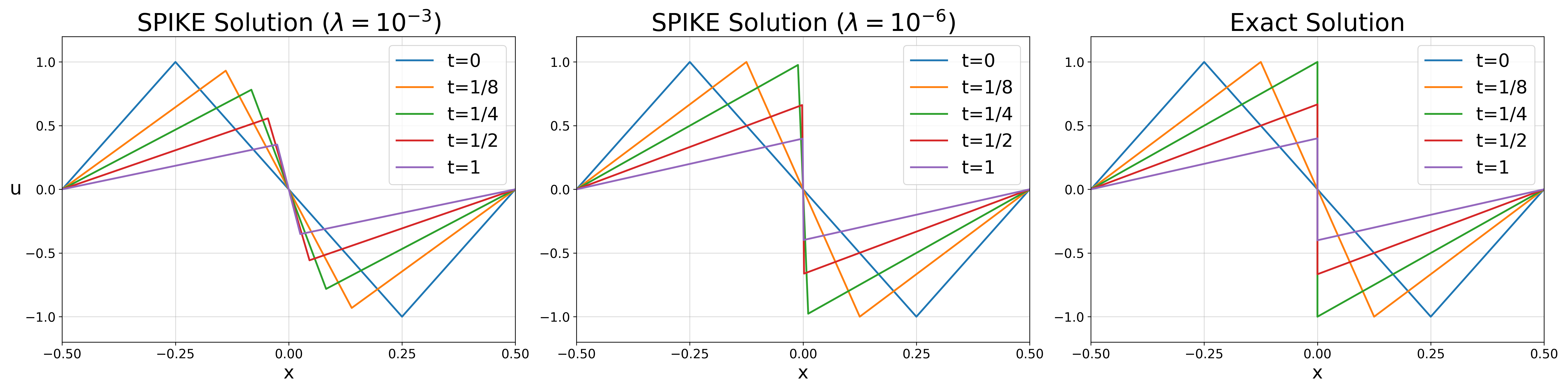}
    \caption{Evolution of the exact zigzag solution (right) and SPIKE approximations with different regularization strengths: $\lambda_a=\lambda_x=10^{-3}$ (left), $10^{-6}$ (middle).}
    \label{fig:zigzag}
\end{figure}

The crucial advantage of this example is that the zigzag profile \eqref{true_zigzag_solution} can be represented exactly within SPIKE's function space using only $N=2$ kernels positioned symmetrically about the origin, as illustrated in Figure \ref{fig:varphi}:
\begin{equation}
    \label{approx_jump}
    a\varphi(x+\delta) - a\varphi(x-\delta) = \begin{cases}
        -H\cdot\frac{x}{\delta},&x\in(-\delta,\delta),\\
        -H\cdot\frac{2x-1}{2\delta-1},&x\in(\delta,1-\delta),
    \end{cases}
\end{equation}
where $H=a\delta(1-2\delta)$. This perfect representability means that any discrepancies between SPIKE and the exact solution must arise purely from regularization effects, not approximation errors.

We parameterize the SPIKE solution as:
\begin{equation*}
    u(x;\theta(t)) = a(t)\varphi\big(x+\delta(t)\big) - a(t)\varphi\big(x-\delta(t)\big),
\end{equation*}
with initial conditions $a(0) = \frac{H_0}{\delta_0(1-2\delta_0)}=8$ and $\delta(0) = \delta_0=\frac{1}{4}$ to match the exact solution at $t=0$.

The SPIKE optimization problem \eqref{SPIKE_EDNN} becomes:
\begin{equation}
    \label{SPIKE_EDNN_zigzag}
    \min_{\dot{a},\dot{\delta}} \int_{\T}\left(\partial_t u + \partial_x(u^2/2)\right)^2 \d x + 2\lambda_a\dot{a}^2 + 2\lambda_x\dot{\delta}^2.
\end{equation}

Through analytical computation, we obtain the parameter evolution dynamics:
\begin{align}
    \label{equ:zigzag_a}
    \dot{a} &= \frac{aH^3(1-4\delta)}{H^3/a + 3\lambda_a(1-6\delta+12\delta^2)a^2 + 3\lambda_xH^2/a^2 + 18\lambda_a\lambda_x}, \\
    \label{equ:zigzag_d}
    \dot{\delta} &= \frac{-H^4/a - 6\lambda_a aH^2}{H^3/a + 3\lambda_a(1-6\delta+12\delta^2)a^2 + 3\lambda_xH^2/a^2 + 18\lambda_a\lambda_x}.
\end{align}

The evolution of $H = a\delta(1-2\delta)$ follows from the chain rule:
\begin{equation}
    \label{equ:zigzag_H}
    \dot{H} = \frac{-6\lambda_a a^2 H^2(1-4\delta)}{H^3/a+3\lambda_a(1-6\delta+12\delta^2)a^2+3\lambda_xH^2/a^2+18\lambda_a\lambda_x}.
\end{equation}

We now analyze the behavior of this system in three scenarios.

\subsubsection{Unregularized Case}

When amplitude regularization is absent ($\lambda_a = 0$), equation \eqref{equ:zigzag_H} immediately yields $\dot{H} = 0$, forcing $H(t) \equiv H_0$ to remain constant throughout the evolution.
This allows us to write $a = H_0/[\delta(1-2\delta)]$, reducing the evolution \eqref{equ:zigzag_d} for $\delta(t)$ to:
\begin{equation*}
    \dot{\delta} = \frac{-H_0}{1+\frac{3\lambda_x}{H_0^2}\delta(1-2\delta)}.
\end{equation*}

Therefore, $\delta(t)$ decreases monotonically and reaches zero at the finite time:
\begin{equation*}
    T = \frac{1}{H_0}\int_0^{\delta_0}\left(1+\frac{3\lambda_x}{H_0^2}\delta(1-2\delta)\right)\d\delta < \infty.
\end{equation*}

As $t \to T^-$, $\delta(t) \to 0$, causing the amplitude $a(t)=\frac{H_0}{\delta(t)(1-2\delta(t))} \to +\infty$. This analysis shows that without amplitude regularization, the parameter evolution may exhibit a finite-time blowup.

\begin{remark}
When all regularization is removed ($\lambda_a = \lambda_x = 0$), the parameter ODE simplifies to $\dot{\delta} = -H_0$ and $\dot{a} = (1-4\delta)a^2/H_0$, which gives the precise evolution equations for the true solution parameters in \eqref{true_zigzag} before shock formation.
This perfect agreement occurs because the true solution lies within SPIKE's representation space and achieves zero PDE residual loss, making it the natural solution to the unregularized optimization problem. The subsequent parameter blowup at $t = 1/4$ is the mathematical signature of shock formation.
\end{remark}

\subsubsection{Regularized Case}

With positive amplitude regularization, the system's behavior alters fundamentally.
We now prove that all parameters remain bounded for all time, preventing the finite-time singularity observed in the unregularized case.
From the evolution equations, $\dot{\delta} \le 0$ ensures $\delta(t) \leq \delta_0 = 1/4$, while $\dot{H} \leq 0$ gives $H(t) \leq H_0 = 1$, and $\dot{a} \geq 0$ yields $a(t) \geq a_0 = 8$.

The crucial step is proving that $a(t)$ remains bounded above, thereby preventing any singularity formation. We begin by bounding the dynamics of $H(t)$ using equation \eqref{equ:zigzag_H}. Since $\delta(t) \in [0, 1/4]$, we find:
\begin{equation*}
    \dot{H}\ge\frac{-6\lambda_a a^2 H^2(1-4\delta)}{3\lambda_a(1-6\delta+12\delta^2)a^2}=\frac{-2(1-4\delta)}{1-6\delta+12\delta^2}\cdot H^2\ge -4H^2.
\end{equation*}
Integrating this differential inequality yields the decay estimate $H(t) \leq 1/(4t+1/H_0)$.

For the amplitude, equation \eqref{equ:zigzag_a} gives:
\begin{equation*}
    \dot{a}\le\frac{aH^3(1-4\delta)}{3\lambda_a(1-6\delta+12\delta^2)a^2}\le\frac{2H^3}{3\lambda_a a}.
\end{equation*}

Multiplying by $a$ and integrating yields an upper bound on the amplitude:
\begin{equation}
    \label{equ:a_bound}
    a(t)^2\le a_0^2+\frac{4}{3\lambda_a}\int_0^t H(s)^3\d s\le a_0^2+\frac{4}{3\lambda_a}\int_0^\infty(4s+1/H_0)^3\d s<\infty.
\end{equation}
Since the integral is finite, $a(t)$ is bounded by a constant of order $O\big(\lambda_a^{-1/2}\big)$. This analysis confirms that positive regularization prevents the finite-time blowup and ensures the global existence of the parameter evolution.

\subsubsection{Vanishing Regularization Limit}
Having established global existence for the regularized case, we investigate the limiting behavior as regularization vanishes. As shown in Figure \ref{fig:zigzag_spike}, the evolution of the SPIKE parameters $H(t)$ and $\delta(t)$ converge to the true, non-smooth solution evolution given by \eqref{true_zigzag}.

This convergence cleanly illustrates the core mechanism of SPIKE: while the true solution exhibits a kink at shock formation, the regularized evolution follows a continuous path that smoothly transitions through this critical point. 
\begin{figure}[ht]
    \centering
    \includegraphics[width=0.8\textwidth]{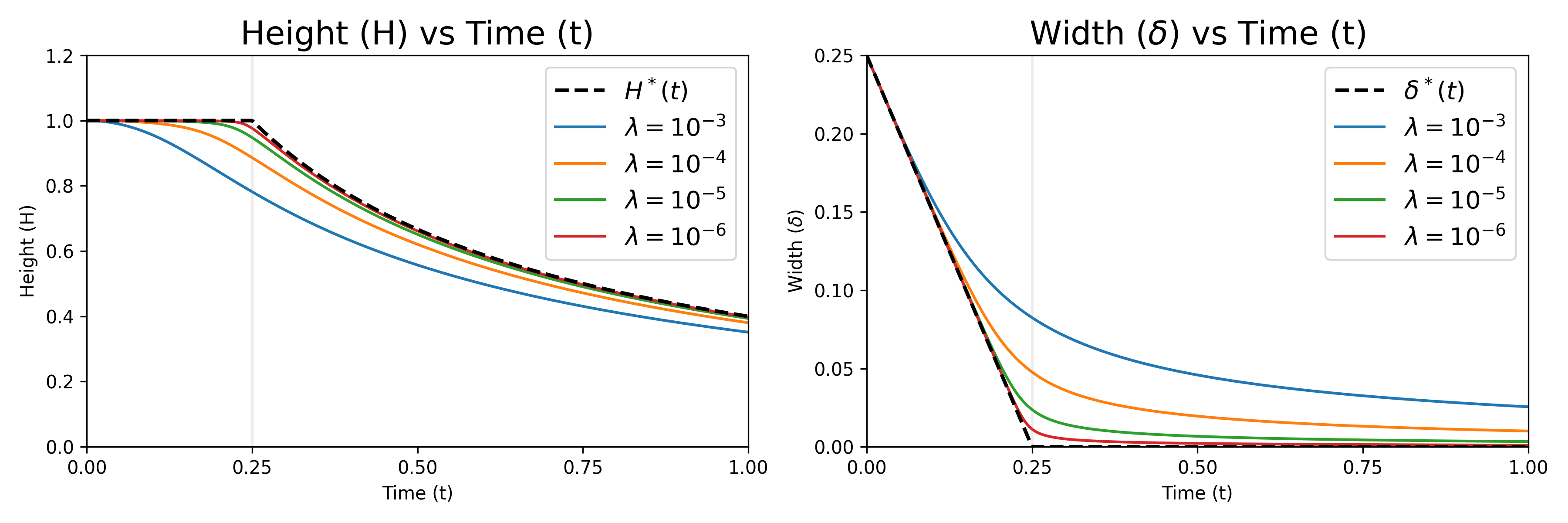}
    \caption{Convergence of SPIKE shape parameters $H(t)$ and $\delta(t)$ to the exact solution (dashed lines) as the regularization strength ($\lambda_a=\lambda_x=\lambda$) vanishes.}
    \label{fig:zigzag_spike}
\end{figure}

\begin{remark}
This analysis reveals a deep structural analogy between SPIKE's regularization and the classical vanishing viscosity method for conservation laws.

In the context of SPIKE, our analysis has established three key behaviors:

\begin{enumerate}[(i)]
\item \textbf{Without regularization}: Parameter evolution exhibits finite-time blowup;
\item \textbf{With positive regularization}: Global solutions exist for all parameters;
\item \textbf{Vanishing regularization limit}: Solutions recover the correct shock dynamics.
\end{enumerate}

These points closely mirror the established theory for viscous conservation laws of the form $\partial_t q + \partial_x f(q) = \nu \partial_{xx} q$:

\begin{enumerate}[(i)]
\item \textbf{Without viscosity}: Strong solutions terminate at shock formation;
\item \textbf{With positive viscosity}: Global strong solutions exist due to dissipation;
\item \textbf{Vanishing viscosity limit}: Solutions approach the entropy solution.
\end{enumerate}

Despite this powerful analogy, the underlying mechanisms are fundamentally different and cannot be made equivalent through algebraic manipulation. 
A key distinction emerges in their long-time behavior:

\textbf{Viscous solutions} inevitably become smoother over time:
\begin{equation*}
    \lim_{t\to\infty}\sup_x |\partial_x u(x,t)| = 0 \quad \text{for any } \nu > 0.
\end{equation*}

\textbf{SPIKE solutions} maintain sharp profiles:
\begin{equation*}
    \lim_{t\to\infty}\sup_x |\partial_x u(x;\theta(t))| = \lim_{t\to\infty} a(t)\big(1-2\delta(t)\big) = O\big(\lambda_a^{-1/2}\big) \quad \text{for } \lambda_a > 0.
\end{equation*}

This demonstrates that SPIKE's regularization operates on the parameter dynamics, not by adding dissipation to the solution itself. Moreover, at the vanishing regularization limit, $\lambda_a \to 0^+$ and $\lambda_a^{-1/2} \to +\infty$, the maintained steepness recovers the sharp discontinuities of the true shock solution. 

The broader relationships between regularization and viscosity mechanisms, as well as whether SPIKE solutions generally satisfy entropy conditions beyond this specific zigzag example, warrant further investigation in future work.
\end{remark}

\subsection{Knot Dynamics and Shock Capturing Mechanism}

The SPIKE method exhibits a profound theoretical property that lies at the heart of resolving our central paradox. A single, unified dynamical system—derived purely from minimizing the strong-form PDE residual with Tikhonov regularization—automatically adapts its behavior based on the local solution structure without any explicit programming of shock physics.

This adaptive behavior manifests in two distinct phases. 
In smooth regions, individual knots track characteristic curves of the hyperbolic system. 
After shocks emerge, they spontaneously reorganize, cluster tightly around discontinuities and collectively propagate at precisely the Rankine-Hugoniot speed. 
The transition between these behaviors occurs naturally through the evolution dynamics, without shock detection algorithms or manual switching between different numerical schemes.

To understand how this seemingly impossible behavior arises, we develop a heuristic analysis of knot dynamics in scalar conservation laws. Our approach examines the vanishing regularization limit ($\lambda_a, \lambda_x \to 0^+$) and the continuum limit ($N \to \infty$), revealing the mathematical mechanism underlying SPIKE's dual personality.

Figure \ref{fig:burgers_knots} provides compelling visual evidence of these phenomena for Burgers' equation with initial condition $u_0(x) = \sin(2\pi x)+0.5$. The solution evolution and knot trajectories clearly show the behavioral transition we analyze below.

\begin{figure}[ht]
    \centering
    \includegraphics[width=0.8\textwidth]{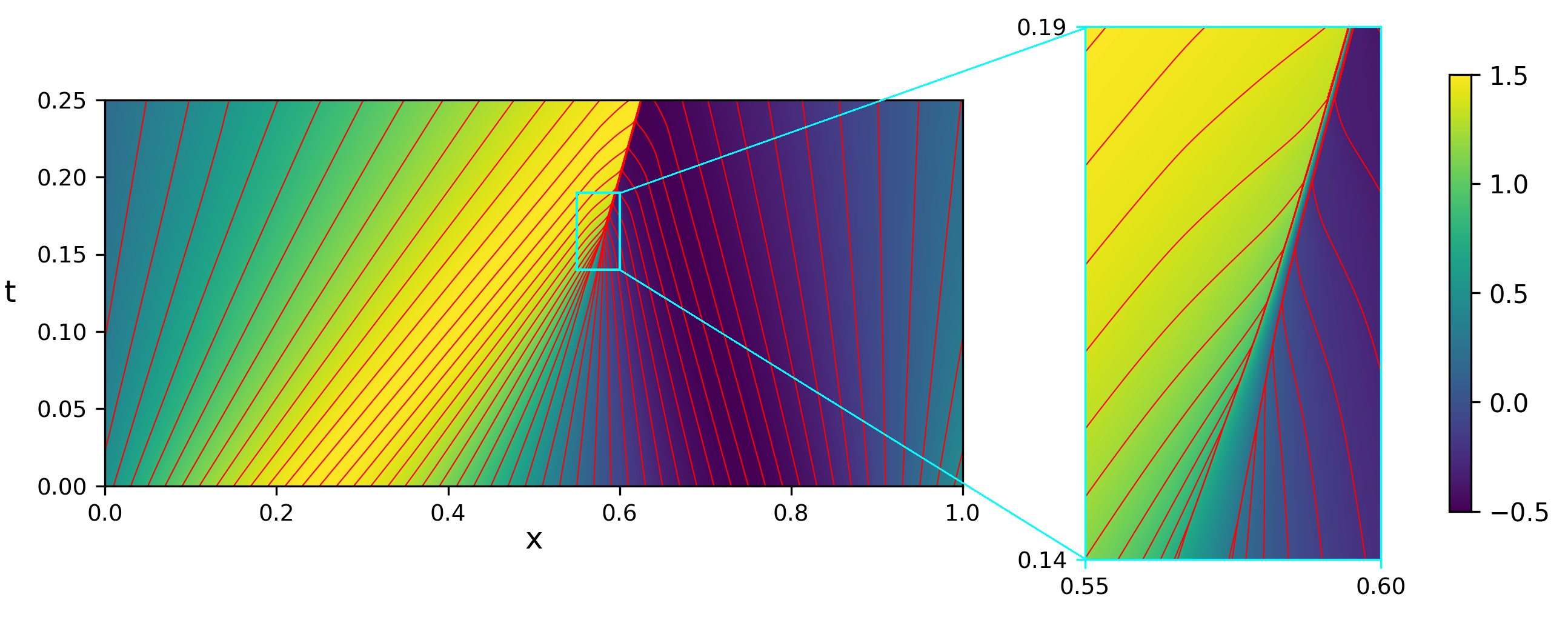}
    \caption{SPIKE solution of Burgers' equation showing the dual nature of knot dynamics. The background contour displays the solution field while red lines trace knot positions over time. Initially, knots move independently along characteristic curves. After shock formation, they spontaneously cluster into two distinct groups that bracket the shock location and move together at the Rankine-Hugoniot speed.}
    \label{fig:burgers_knots}
\end{figure}

\subsubsection{Phase I: Characteristic Following in Smooth Regions}

In the vanishing regularization limit ($\lambda \to 0^+$), the optimality conditions \eqref{opt_a_x_F} therefore require $\mathfrak{D}_i = \mathfrak{D}'_i = 0$, which immediately implies $\mathfrak{h}_i=\mathfrak{F}_i$ and $\mathfrak{h}_i'=\mathfrak{F}_i'$ throughout the domain. Substituting these conditions into \eqref{hDDjump_expanded} yields:
\begin{align}
    \label{hFjump}
    -a_i\dot{x}_i &= \bigg(\frac{6(\mathfrak{F}_{i+1}-\mathfrak{F}_i)}{\delta_{i_+}^2} - \frac{(4\mathfrak{F}'_{i}+2\mathfrak{F}'_{i+1})}{\delta_{i_+}}\bigg) - \bigg(\frac{6(\mathfrak{F}_{i-1}-\mathfrak{F}_{i})}{\delta_{i_-}^2} + \frac{(4\mathfrak{F}'_{i}+2\mathfrak{F}'_{i-1})}{\delta_{i_-}}\bigg).
\end{align}

Since the SPIKE solution $q\big(x;\theta(t)\big)$ is piecewise linear and the flux $f$ is smooth, the function $\mathfrak{F}$, with $\mathfrak{F}' = f(q)-\overline{f(q)}$, is also smooth between knots. 
We can therefore apply the Taylor expansion around $x_i$:
\begin{align*}
    \mathfrak{F}_{i\pm1} &= \mathfrak{F}(x_i^\pm) \pm \mathfrak{F}'(x_i^\pm)\cdot\delta_{i_\pm} + \mathfrak{F}''(x_i^\pm)\cdot\delta_{i_\pm}^2/2 \pm \mathfrak{F}'''(x_i^\pm)\cdot\delta_{i_\pm}^3/6 + O(\delta_{i_\pm}^4),\\
    \mathfrak{F}'_{i\pm1} &= \mathfrak{F}'(x_i^\pm) \pm \mathfrak{F}''(x_i^\pm)\cdot\delta_{i_\pm} + \mathfrak{F}'''(x_i^\pm)\cdot\delta_{i_\pm}^2/2 + O(\delta_{i_\pm}^3).
\end{align*}

Since $\mathfrak{F}''(x) = f'\big(q(x)\big)\cdot \partial_x q(x)$, substituting these expansions into \eqref{hFjump} yields:
\begin{equation*}
    -a_i\dot{x}_i=f'\big(q(x_i)\big)\cdot\big(\partial_x q(x_i^+) - \partial_x q(x_i^-)\big) + O(\delta_{i_+}^2 + \delta_{i_-}^2).
\end{equation*}

Recalling from \eqref{eq:amp_second_derivative} that $a_i = -\partial_x q(x_i^+)+\partial_x q(x_i^-)$, we obtain:
\begin{equation*}
    \dot{x}_i = f'\big(q(x_i)\big) + O(\delta_{i_+}^2 + \delta_{i_-}^2) / a_i.
\end{equation*}

In the continuum limit $N\to\infty$, the spacing $\delta_{i_\pm}$ and the amplitude $a_i$ typically scale as $O(N^{-1})$, resulting in:
\begin{equation*}
    \dot{x}_i = f'\big(q(x_i)\big) + O(1/N).
\end{equation*}

This fundamental result reveals the first part of SPIKE's dual nature. The knots automatically move at the characteristic speed of the scalar conservation law, despite the fact that we never explicitly programmed the method of characteristics.

\subsubsection{Phase II: Collective Shock Propagation}

As characteristics converge, the knots following them naturally cluster and form tightly-packed groups on both sides of the discontinuity. 
Motivated by the trajectories observed in Figure \ref{fig:burgers_knots}, we model this clustered configuration using an effective two-kernel representation:
\begin{equation}  
    \label{eq:shock_model}  
    q\big(x;\theta(t)\big) = \underbrace{a(t)\varphi\big(x-x_s(t)+\delta(t)\big) - a(t)\varphi\big(x-x_s(t)-\delta(t)\big)}_{\text{effective shock representation}} + \underbrace{\tilde{q}\big(x;\tilde{\theta}(t)\big)}_{\text{smooth background}}.
\end{equation}

Here, $x_s(t)$ represents the shock position, the parameters $a(t)$ and $\delta(t)$ define the jump $H = a\delta(1-2\delta)$, and the smooth background term $\tilde{q}\big(x;\tilde{\theta}(t)\big)$ accounts for the contribution from other kernels far from the shock.

In the vanishing regularization limit, \eqref{SPIKE_EDNN} simplifies to pure residual minimization:  
\begin{equation}  
    \label{eq:shock_opt}  
    \min_{\dot{a},\dot{x}_s,\dot{\delta},\dot{\tilde{\theta}}} \int_{\T} \left|\partial_t q + \partial_x f(q)\right|^2 \d x,
\end{equation}  
where the time derivative of the effective model \eqref{eq:shock_model} is given by:
\begin{equation}  
    \label{eq:shock_time_derivative}  
    \begin{aligned}  
    \partial_t q \approx\;& \dot{a}\big(\varphi(x-x_s+\delta)-\varphi(x-x_s-\delta)\big) + a\dot{\delta}\big(\varphi'(x-x_s+\delta)+\varphi'(x-x_s-\delta)\big)\\  
    &-a\dot{x}_s\big(\varphi'(x-x_s+\delta)-\varphi'(x-x_s-\delta)\big) + \partial_t\tilde{q}.
    \end{aligned}  
\end{equation}

The first-order optimality condition with respect to the shock speed $\dot{x}_s$ is:
\begin{equation}
    \label{eq:opt_cond_shock}
    \begin{aligned}
        0=&\frac{\partial}{\partial \dot{x}_s} \int_{\T} \left|\partial_t q + \partial_x f(q)\right|^2 \d x = \int_{\T} 2\,\frac{\partial(\partial_t q)}{\partial \dot{x}_s}\cdot\big(\partial_t q + \partial_x f(q)\big)\d x\\
        =&\int_{\T} -2a\big(\varphi'(x-x_s+\delta)-\varphi'(x-x_s-\delta)\big)\cdot\big(\partial_t q + \partial_x f(q)\big)\d x.
    \end{aligned}
\end{equation}

From definition of $\varphi$ in \eqref{def_varphi}, we have $\varphi'(x) = \{x\} - 1/2$, which gives:
\begin{equation*}
    \varphi'(x-x_s+\delta)-\varphi'(x-x_s-\delta) = \begin{cases}
        2\delta-1, & x \in (x_s-\delta, x_s+\delta),\\
        2\delta, & \text{otherwise}.
    \end{cases}
\end{equation*}

Substituting into the optimality condition \eqref{eq:opt_cond_shock} gives:
\begin{equation*}
    2\delta\int_{\T}\big(\partial_t q + \partial_x f(q)\big)\d x - \int_{x_s-\delta}^{x_s+\delta} \big(\partial_t q + \partial_x f(q)\big)\d x = 0.
\end{equation*}

Due to the conservation property established in Section \ref{sec:conservation}, the first integral vanishes, leaving:
\begin{equation}
    \label{eq:shock_integral_condition}
    \int_{x_s-\delta}^{x_s+\delta} \big(\partial_t q + \partial_x f(q)\big)\d x = 0
\end{equation}

Substituting the time derivative \eqref{eq:shock_time_derivative} and evaluating this integral over the narrow shock region by the Newton-Leibniz formula, we obtain:
\begin{equation*}
    2a\delta(1-2\delta)\cdot\dot{x}_s + f(q)\Big|_{x_s-\delta}^{x_s+\delta} + \int_{x_s-\delta}^{x_s+\delta} \partial_t\tilde{q}\,\d x = 0.
\end{equation*}

In the sharp shock limit where $|a| \to \infty$ and $\delta \to 0$ while maintaining finite jump magnitude $H = a\delta(1-2\delta)$, the last integral involving the smooth background $\tilde{q}$ becomes negligible. 
This yields the classical Rankine-Hugoniot condition:  
\begin{equation*}  
    f\big(q(x_s^+)\big) - f\big(q(x_s^-)\big) = \big(q(x_s^+) - q(x_s^-)\big)\cdot \dot{x}_s.
\end{equation*}

This analysis provides a complete resolution to our central paradox. The reproducing kernel $\varphi$ acts as a mathematical translator: its reproducing property ensures that domain-wide integrals of the strong-form residual in \eqref{eq:shock_opt} automatically concentrate on regions with sharp gradients, as illustrated by \eqref{eq:shock_integral_condition}. This emergent property reveals a deep and previously unexplored connection between variational physics-informed methods and the theory of hyperbolic conservation laws.

\section{Numerical Experiments}
\label{sec:experiments}
This section evaluates SPIKE's shock-capturing performance across various conservation laws, from scalar equations to vector-valued systems. 
Our experiments demonstrate the method's universal effectiveness without problem-specific modifications, using the same minimal formulation (typically $N=500$ kernels, $\lambda_a = \lambda_x = 10^{-7}$). We compare against Physics-Informed Neural Networks (PINN) \cite{raissi2019physics} and Evolutional Deep Neural Network (EDNN) \cite{EDNN}, with high-resolution finite-volume solutions as references.

\subsection{Burgers' Equation}
The inviscid Burgers' equation $$\partial_t u + \partial_x(u^2/2) = 0$$ provides an ideal testing ground due to its simple nonlinearity that generates complex shock interactions. We examine the initial condition $u_0(x)=\cos(2\pi x)-\cos(4\pi x)$, which creates two independent shocks that propagate, collide, and merge into a single stronger shock.

Figure \ref{fig:bimode_results} reveals dramatic performance differences. SPIKE produces remarkably sharp shock profiles that closely match the reference solution throughout shock formation, propagation, and interaction without oscillations. 
The error remains tightly localized around discontinuities, indicating precise shock resolution. 
In contrast, both competing methods exhibit fundamental difficulties after shock formation. PINN's space-time neural network severely over-smooths discontinuities, while EDNN generates widespread oscillatory artifacts that contaminate the entire solution domain. 

\begin{figure}[ht]
    \centering
    \includegraphics[width=\textwidth]{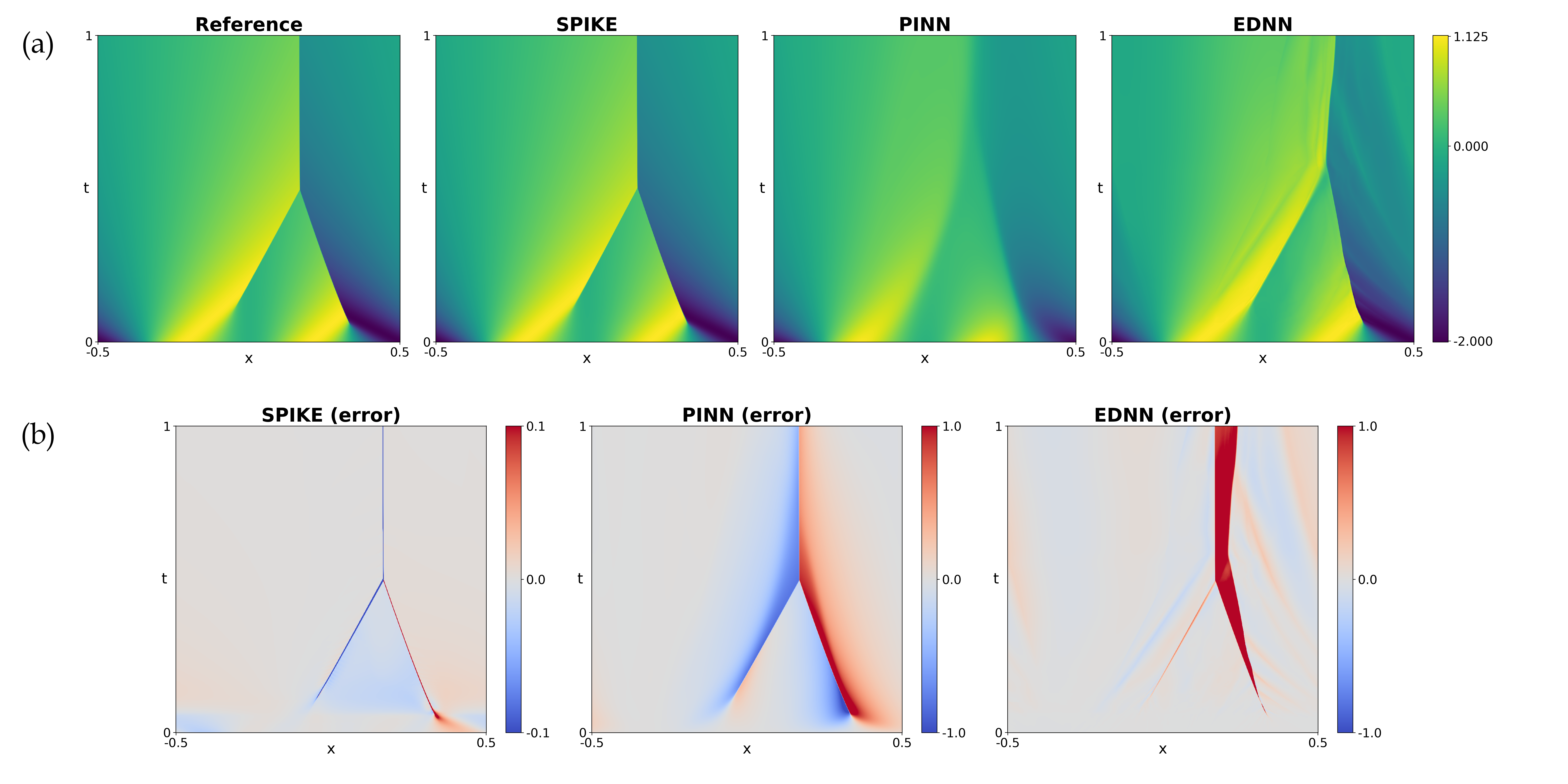}
    \caption{(a) Solution contours of Burgers' equation using SPIKE, PINN, and EDNN, compared to a high-resolution reference. (b) Pointwise error relative to the reference solution.}
    \label{fig:bimode_results}
\end{figure}

\subsection{Buckley-Leverett Equation}

We next examine the Buckley-Leverett equation that models two-phase flow in porous media: $$\partial_t u + \partial_x \big(\frac{u^2}{u^2+(1-u)^2}\big) = 0.$$
This problem presents greater complexity than Burgers' equation through its non-convex flux function. The initial condition $u_0(x) = \sin^4(\pi x)$ creates two shocks that propagate with varying speeds.

Figure \ref{fig:BL_results} confirms SPIKE's effectiveness extends to non-convex problems. The method preserves its shock-capturing accuracy without requiring flux-specific modifications, demonstrating the generality of the approach. The error patterns remain consistent: SPIKE produces narrow error bands at discontinuities while PINN and EDNN exhibit global solution contamination.

\begin{figure}[ht]
    \centering
    \includegraphics[width=\textwidth]{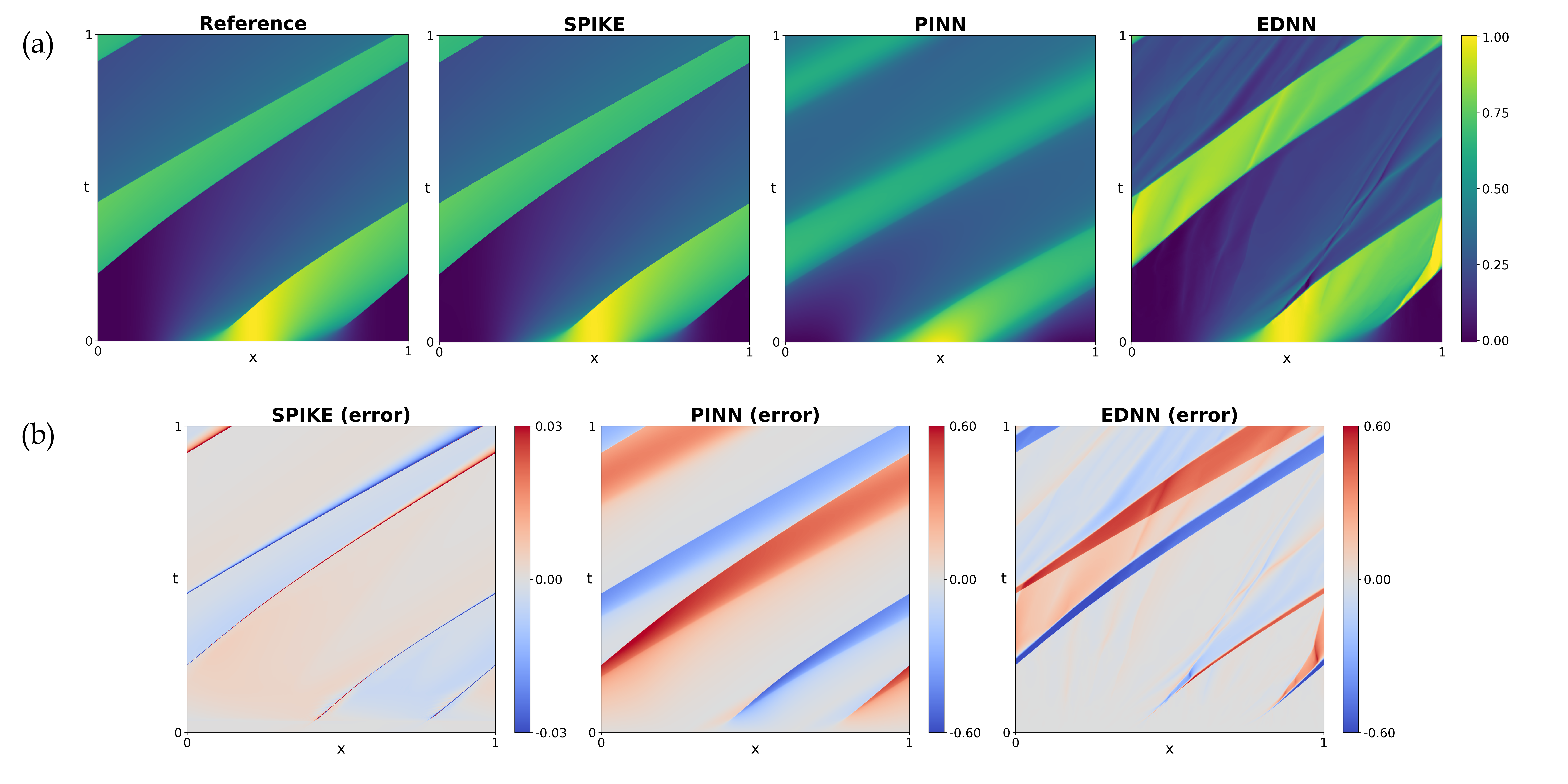}
    \caption{(a) Solution contours of the Buckley-Leverett equation using SPIKE, PINN, and EDNN, compared to a high-resolution reference. (b) Pointwise error relative to the reference solution.}
    \label{fig:BL_results}
\end{figure}

\subsection{Euler Equations}
Although the theoretical analysis in Section \ref{sec:analysis} focused on scalar conservation laws, the SPIKE framework naturally extends to vector-valued systems. To demonstrate this capability, we consider the Euler equations for inviscid compressible flow:
\begin{equation*}
    \partial_t\begin{bmatrix}
        \rho\\
        \rho v\\
        \rho E
    \end{bmatrix} + \partial_x\begin{bmatrix}
        \rho v\\
        \rho v^2 + p\\
        (\rho E + p)v
    \end{bmatrix} = 0,
\end{equation*}
where $\rho$, $v$, $p$, and $E = \frac{1}{2}v^2 + e(p,\rho)$ denote density, velocity, pressure, and total energy, respectively. The internal energy is given by $e(p,\rho) = \frac{p}{(\gamma - 1)\rho}$ with the adiabatic index $\gamma=1.4$. We initialize the system with smooth data that rapidly develops into complex shock patterns: $\rho_0(x) = 1, v_0(x) = 1, p_0(x) = 2 + \sin(2\pi x)$.

This problem presents a particularly demanding test case, as it involves the formation and interaction of multiple shock waves over an extended time horizon ($T = 5.0$). Such long-time simulations often reveal numerical instabilities and resolution issues that remain hidden in shorter runs. Given the substantial difficulties that PINN and EDNN exhibit even for scalar problems, we focus this evaluation exclusively on SPIKE's performance against a high-fidelity reference solution.

The adaptive clustering behavior observed in Figure \ref{fig:burgers_knots} is essential for shock resolution but may become problematic over long times, leaving regions with insufficient knot coverage. We therefore monitor the distribution quality via the effective knot count $N_{\text{eff}} := 1\big/{\sum_i |x_{i+1} - x_i|^2}$, which quantifies how uniformly the knots remain distributed: perfectly uniform spacing yields $N_{\text{eff}} = N$, while extreme clustering gives $N_{\text{eff}} \approx 1$.
When clustering reduces $N_{\text{eff}}$ below a critical threshold (e.g., $0.6N$), we redistribute the knots uniformly across the domain and reinitialize the amplitudes using the procedure from Section \ref{sec:init}. This redistribution strategy augments the method's adaptive capabilities over long time horizons.

The solution contours in Figure \ref{fig:euler_long_term} provide compelling evidence of SPIKE's robustness for complex vector systems. The method successfully accurately captures the intricate evolution of multiple interacting shocks while maintaining numerical stability. 
As shown in Figure \ref{fig:euler_error}, the $L^1(\T)$ error for each variable remains bounded and well-controlled throughout the entire simulation duration.
These results demonstrate SPIKE's effectiveness for challenging vector conservation laws and highlight its potential for demanding, long-term fluid dynamics applications.

\begin{figure}[ht]  
    \centering
    \includegraphics[width=\textwidth]{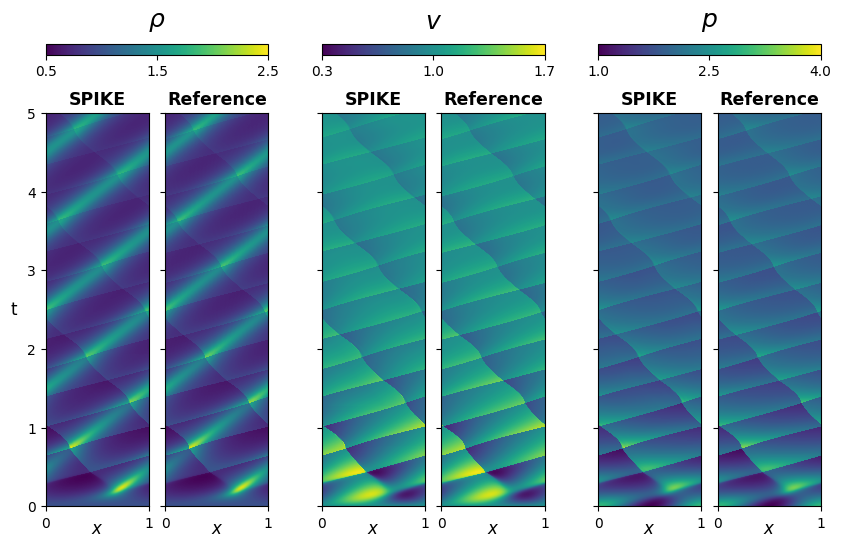}  
    \caption{Long-term evolution of the 1D Euler equations using SPIKE, compared to a high-resolution reference solution.}
    \label{fig:euler_long_term}
\end{figure}

\begin{figure}[ht]  
    \centering  
    \includegraphics[width=0.75\textwidth]{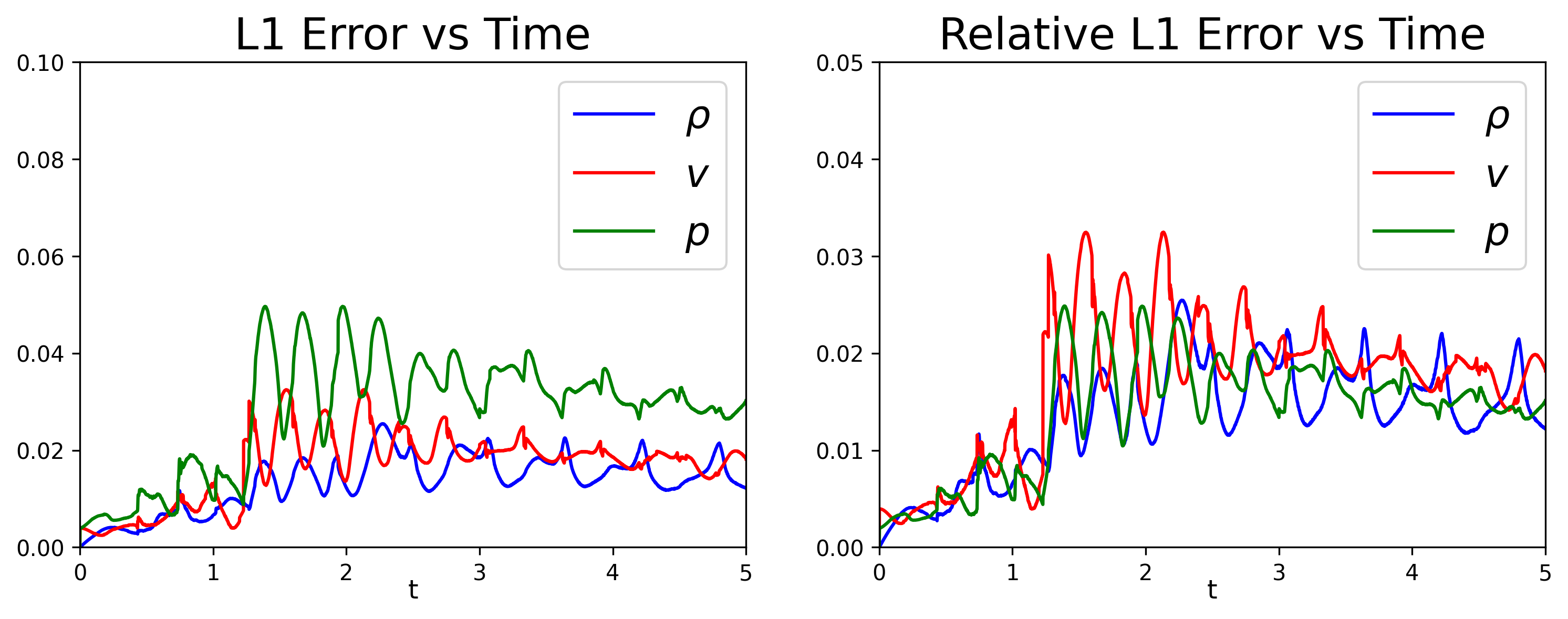}
    \caption{Time evolution of the $L^1(\T)$ error and the relative error for each variable in the Euler equations.}
    \label{fig:euler_error}
\end{figure}

\section{Conclusions}

This paper introduces SPIKE (Stable Physics-Informed Kernel Evolution), a novel method for solving inviscid hyperbolic conservation laws that captures discontinuous weak solutions through strong-form PDE residual minimization.

Through dynamic kernel superposition with evolving positions, SPIKE exhibits adaptive behavior: kernels track characteristics in smooth regions and cluster near discontinuities to propagate shocks at the Rankine-Hugoniot speed. This dual behavior emerges naturally from the regularized optimization process without requiring explicit shock detection or artificial viscosity terms.

The regularization mechanism prevents finite-time parameter blowup while ensuring convergence to the correct weak solution in the vanishing regularization limit. This provides insight into how Tikhonov regularization serves as a smooth transition mechanism through shock formation, allowing optimization dynamics to traverse discontinuities that appear singular from the strong-form perspective.

By exploiting the kernel's spline structure, the method achieves linear computational complexity during evolution. Numerical experiments across scalar and vector conservation laws demonstrate SPIKE's superior performance compared to existing physics-informed methods, maintaining sharp discontinuities with controlled error growth over extended simulations.

These results establish previously unexplored connections between reproducing kernel theory and hyperbolic conservation laws, suggesting that kernel-based approaches offer fundamental advantages for problems involving discontinuous solutions.

Beyond SPIKE's demonstrated success with one-dimensional hyperbolic conservation laws, several theoretical and computational questions emerge from this approach. The regularization mechanism operates beyond standard Hilbert space settings, requiring new analytical frameworks to understand how weak solution behavior emerges from strong-form optimization. Additionally, fundamental concepts from traditional numerical methods, like CFL conditions and convergence rate dependencies on $N$ and $\lambda$, remain unclear for this framework. Practical developments include extending to different domain configurations and implementing more sophisticated adaptive knot management strategies. Addressing these theoretical and computational challenges would broaden the foundation and utility of this approach for hyperbolic PDE.

\section*{Acknowledgement}
We thank Prof. Xiaochuan Tian and Dr. He Zhang for helpful discussions.

\bibliographystyle{siamplain}
\bibliography{SPIKEref}

\end{document}